%% file: ms.tex
\documentclass[onefignum,onetabnum]{siamart171218}

\input{ex_shared}

\ifpdf
\hypersetup{
  pdftitle={Shape Optimization using the Finite Element Method on Multiple Meshes},
  pdfauthor={J.S. Dokken}
}
\fi

\begin{document}

\maketitle

\begin{abstract}
  An important step in shape optimization with partial differential equation
  constraints is to adapt the geometry during each optimization iteration.
  Common strategies are to employ mesh-deformation or re-meshing, where one or the other typically lacks robustness or is computationally expensive. This
    paper proposes a different approach, in which the computational domain is
    represented by multiple, independent meshes.
    A Nitsche based finite element method is used to weakly enforce continuity
    over the non-matching mesh interfaces. The optimization is preformed using
    an iterative gradient method, in which the shape-sensitivities are obtained
    by employing the Hadamard formulas and the adjoint approach.
    An optimize-then-discretize approach is chosen due to its independence
    of the FEM framework. Since the individual meshes may be moved freely,
    re-meshing or mesh deformations
    can be entirely avoided in cases where the geometry changes consists of
    rigid motions or scaling. By this free movement, we obtain robust and
    computational cheap mesh adaptation for optimization problems even for
    large domain changes. For general geometry changes, the method can be
    combined with mesh-deformation or re-meshing techniques to reduce the amount
    of deformation required.  We demonstrate the capabilities of the method
    on several examples, including the optimal placement of heat emitting wires
    in a cable to minimize the chance of overheating, the drag
    minimization in Stokes flow, and the orientation of 25 objects in a Stokes flow.
\end{abstract}



\input{./introduction_draft.tex}

\input{./algorithm.tex}

\input{./mmfem.tex}

\input{./optimization.tex}

\input{./numerical_examples.tex}

\input{./conclusion.tex}


\section*{Acknowledgments}
The authors would like to acknowledge all FEniCS developers for
feedback, and especially those involved in the MultiMesh project.

\bibliographystyle{siamplain}
\bibliography{references}
\end{document}

%% file: ex_shared.tex
\usepackage{lipsum}
\usepackage{amsfonts}
\usepackage{graphicx}
\usepackage{epstopdf}
\usepackage{algorithmic}
\ifpdf
  \DeclareGraphicsExtensions{.eps,.pdf,.png,.jpg}
\else
  \DeclareGraphicsExtensions{.eps}
\fi

\newcommand{\norm}[1]{\left\lVert#1\right\rVert} 

\newsiamremark{remark}{Remark}
\newsiamremark{hypothesis}{Hypothesis}
\crefname{hypothesis}{Hypothesis}{Hypotheses}
\newsiamthm{claim}{Claim}

\headers{Shape Optimization using MultiMesh FEM with Nitsche coupling}{J. S. Dokken, S. W. Funke, A. Johansson, and S. Schmidt}

\title{Shape Optimization using the Finite Element Method on Multiple Meshes with Nitsche coupling\thanks{Submitted to the editors DATE.}
\funding{This work was supported by the Research Council of Norway through a FRIPRO grant, project number 251237.}}

\author{J\O rgen S. Dokken\thanks{Simula Research Laboratory, Lysaker, Norway
  (\email{dokken@simula.no}).}
  \and Simon W. Funke \footnotemark[2]
  \and August Johansson \footnotemark[2]
  \and Stephan Schmidt \thanks{Universit\"{a}t W\"{u}rzburg, Germany}.}

\usepackage{amsopn}

\usepackage{tikz}
\crefname{section}{Section}{Sections}
\Crefname{algocf}{Algorithm}{Algorithms}
\crefname{equation}{Equation}{Equations}
\crefname{table}{Table}{Tables}
\crefname{figure}{Figure}{Figures}

\newcommand{\Rb}{\mathbb{R}}
\newcommand{\example}[1]{\textbf{Example:} {\itshape #1}}
\newcommand{\examplecont}[1]{\textbf{Example (cont.):} {\itshape #1}}
\newcommand\Int[2]{\int\limits_{ #1}^{#2}} 
\newcommand\der[2]{\frac{\partial #1}{\partial #2}} 
\newcommand\totder[2]{\frac{\mathrm{d}#1}{\mathrm{d} #2}} 
\newcommand{\md}{\ \mathrm{d}} 
\newcommand{\dOmega}{\partial\Omega}
\usepackage{pdfpages}
\graphicspath {{.}}
\newcommand{\inv}[1]{#1^{-1}} 
\newcommand{\half}{\frac{1}{2}}

\usepackage{stmaryrd}

\usepackage[section]{placeins}


%% file: introduction_draft.tex
\section{Introduction}
During the last two decades, there has been a transition from simulation to
coupling of optimization and simulation~\cite{tekin2004simulation}.
Of particular industrial relevance are shape optimization problems, which aim
to optimize the shape of an object subject to physical constraints, typically
described by partial differential equations (PDEs).
Examples of industrial problems that have been modeled are the drag minimization of airplanes and cars
~\cite{mohammadi2010applied,nielsen2006adjoint,reuther1999constrained}, the shape optimization of acoustic horns~\cite{schmidt2016large}, and the optimal design of current carrying multi-cables~\cite{harbrecht2016}.
The success of these applications is driven by efficient optimization algorithms and fast methods for solving PDEs.
More specifically, gradient-based optimization methods have shown to converge quickly and often independent of the number of design variables. The required shape gradients are derived through shape calculus and the adjoint PDE~\cite{delfour2011shapes,schmidt2013three,sokolowski1992introduction}. The Finite Element Method (FEM) is an efficient and flexible method for solving a wide range of PDEs. In the last decades, this method has gained popularity in both the scientific and industrial environment due to its mathematical foundation and geometrical flexibility.

A critical part in shape optimization algorithms is handling of geometry changes during each optimization iteration. For FEM based models this means that the computational mesh must be updated to a new target geometry at low cost while maintaining a high mesh quality.
Two commonly used mesh updating strategies are mesh deformation and re-meshing.
Mesh deformation methods often involve the solution of an auxiliary PDE. However, the mesh quality may degrade or even degenerate for large deformations. Several deformation schemes have therefore been proposed to handle large deformations~\cite{schmidt2014two,stein2003mesh} of the expense of a high computational cost.
In contrast, a re-meshing strategy produces meshes that are guaranteed to be regular for any geometrical change. However, drawbacks are that the geometry must be reconstructed from the mesh to allow for re-meshing of the boundary elements, and the high computational cost of the meshing algorithms~\cite{boggs2005dart}.

To overcome these limitations, we propose a shape optimization algorithm based on the idea to represent the domain by multiple, non-matching meshes, as illustrated in \cref{fig:Ball}. Because our method is highly embedded in the finite element setting, we call it MultiMesh as opposed to existing approaches like Chimera and Overset methods. Each mesh can be freely rotated, scaled or translated at a low computational cost without impacting the mesh quality.
In an optimization setting, this means that mesh updates can be completely eliminated in cases where the goal is to optimally rotate, scale or translate of objects within a larger geometry. Further, as we will show in the numerical examples, the MultiMesh approach is beneficial for general shape optimization. Applying mesh deformation or re-meshing to a MultiMesh is computationally cheaper and more robust than the traditional single mesh approach, since it is only applied on submeshes. To the best of the authors' knowledge, this is the first instance of a FEM with multiple overlapping meshes in the setting of shape optimization.

\begin{figure}[!ht]
  \centering
  \begin{tikzpicture}
    \node (sm) [above right, inner sep=0pt]{
      \includegraphics[width=\linewidth]{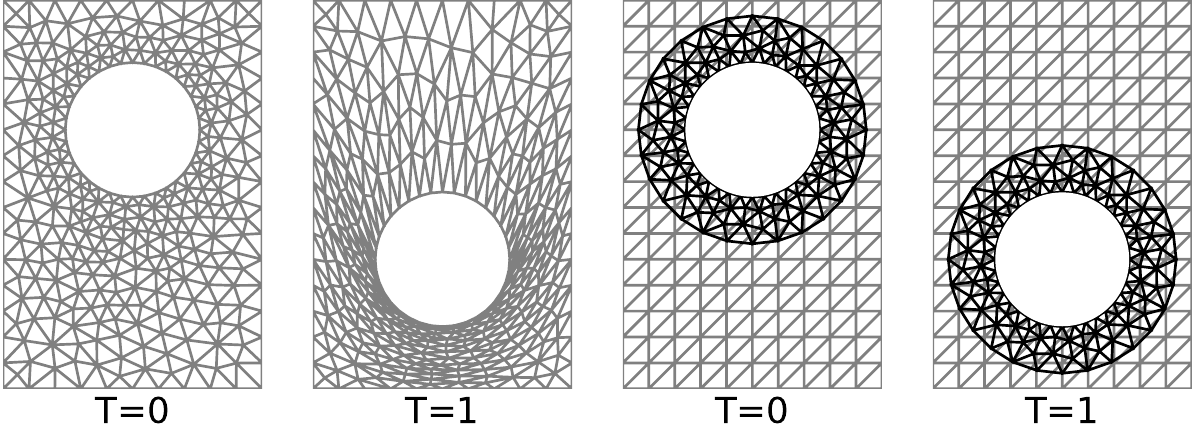}};
    \node (A) at (0.27\linewidth,-0.2) {(a)};
    \node (B) at (0.775\linewidth,-0.2) {(b)};
  \end{tikzpicture}

  \caption{A comparison of a moving object described with a standard mesh and
    with multiple meshes. In (a), the mesh is deformed with an
    Eikonal convection equation, combined with a centroidal Voronoi
    tessellation (CVT) smoothing~\cite{schmidt2014two}. The mesh quality,
    quantified by the minimum radius ratio decreases from $0.75$
    to $6\cdot 10^{-4}$, and the mesh is degenerated.
    In (b), a mesh describing the ball is
    introduced, and can be translated independent of the background mesh. Here
    the minimum radius ratio is constant at $0.72$.
    }\label{fig:Ball}
\end{figure}

The use of multiple meshes dates back to solving the problem of structure mesh generation for finite difference or structured finite volume schemes~\cite{benek1983flexible,henshaw2002overture,starius1977composite,volkov1970method}.
These many-mesh techniques (also known as Chimera or Overset techniques) overcome several limitations of structured grids, such as multiple holes and moving domains, making them  particularly popular for aerodynamic applications~\cite{stangl1996euler}. 
These schemes have also been used in an optimization setting, with similar data-transfer for the adjoint equation~\cite{liao2006aerodynamic}.

A recent method for non-matching meshes for FEM is the
Cut Finite Element Method (CutFEM)~\cite{burman2015cutfem}.
This method uses a Nitsche based formulation to weakly enforce boundary
conditions at the interface. 
CutFEM has been used for a wide range of shape and topology optimization problems, such as acoustics~\cite{bernlandacoustic}, elasticity~\cite{bandara2016shape,burman2018shape} and incompressible flow~\cite{villanueva2017cutfem}. The MultiMesh FEM~\cite{MMFEMarbit} is a generalization of the CutFEM, where the computational domain is described by arbitrary many overlapping non-matching meshes,
coupled with Nitsche's method. The MultiMesh FEM has been explored for the Poisson and Stokes-equations~\cite{hansen2016simulation, johansson2013high, johansson2015high}.
Other methods used for shape optimization of complex computational domains are available, see for example~\cite{NAJAFI20151, noel2017shape, van2005generalized} and the references therein.

The MultiMesh FEM introduces several interesting
aspects when applied to a shape optimization problem.
When creating a solution algorithm for the optimization problem, a choice has to
be made, namely, should one use the \emph{first optimize, then discretize approach} or
the first discretize, then optimize approach. An initial analysis of these approaches has been investigated with respect to shape optimization problems~\cite{berggren}. However, there is no general recipe for
which method is to be preferred~\cite{hinze2008optimization}.
Due to generality and independence of the FEM framework, we have chosen the optimize then discretize approach. However, the authors plan to investigate the effects of a discretize then optimize strategy in a later publication.
Numerical examples show that the numerical inconsistency in the shape gradient is insignificant for fine meshes.

The remainder of this paper is organized as follows. \cref{sec:math} introduces the mathematical notation and presents the new algorithm for solving shape optimization on multiple domains.
\cref{sec:MMFEM} gives a brief introduction solving PDEs on multiple meshes with the MultiMesh FEM.
In \cref{sec:shape}, we then derive shape derivatives using shape calculus, and the associated adjoint equations.
\cref{sec:updatemesh} discusses the optimization step and the mesh updating strategies.
Thereafter, we present several numerical examples in \cref{sec:numerical_examples} and compare the new approach to a traditional FEM when feasible. Finally, we summarize and draw conclusions in \cref{sec:conc}.




%% file: algorithm.tex
\section{Algorithm for solving shape optimization on multiple overlapping domains}\label{sec:math}
In this section, we present the algorithm for solving PDE constrained shape optimization problems using the MultiMesh Finite Element Method (FEM).

In this paper, we consider shape optimization problems of the form
\begin{align}
  &\min_{u,\Omega} J(u,\Omega)\label{eq:GeneralFunctional}\\
  &\text{ subject to}\nonumber\\
  &E(u, \Omega)=0 \label{eq:GeneralState}
\end{align}

where $J(u, \Omega)\in \mathbb R$ is an objective functional, $E(u, \Omega)=0$ is a PDE with solution $u$ defined over the polynomial domain $\Omega\subset\Rb^n, n=1,2,3$.

A related and common case is where the domain $\Omega$ is parameterized and the goal is to optimize the design parameters. The following example illustrates this.

\example{
Consider the problem of minimizing the squared $L_2$-norm of the solution of the Poisson equation.
The domain contains an obstacle, which may be rotated freely around a point $p$. 
The objective is to determine the optimal rotation angle $\theta$ of the obstacle. 
The optimization problem takes the form
\begin{align}
  \min_{T,\theta} J(T,\theta)=\min_{T,\theta} \Int{\Omega(\theta)}{}
  T^2\md x\label{eq:SimpleFunctional},
\end{align}
subject to
\begin{align}
  \begin{split}\label{eq:SimpleState}
    -\Delta T &= f \quad\text{ in } \Omega,\\
    T &= g \quad \text{ on } \dOmega
    ,\\
    g &= \begin{cases}
      1 \quad\text{ on } \Gamma\subset\partial\Omega,\\
      0 \quad\text{ on } \dOmega\setminus\Gamma.
     \end{cases}
  \end{split}
\end{align}
Here, the domain $\Omega$ consists of a rectangle with an elliptic obstacle parameterized by $\theta$. Further, the boundary of $\Omega$ including the boundary of the obstacle is denoted by $\partial \Omega$. The boundary of the obstacle is denoted by $\Gamma\subset\partial\Omega$ and $f(x, y)=x \sin(x) \cos(y)$ is the source function.
The setup is visualized in \cref{fig:SinglePoisson}.
}

\begin{figure}[!ht]
  \centering
  \includegraphics[width=0.5\linewidth]{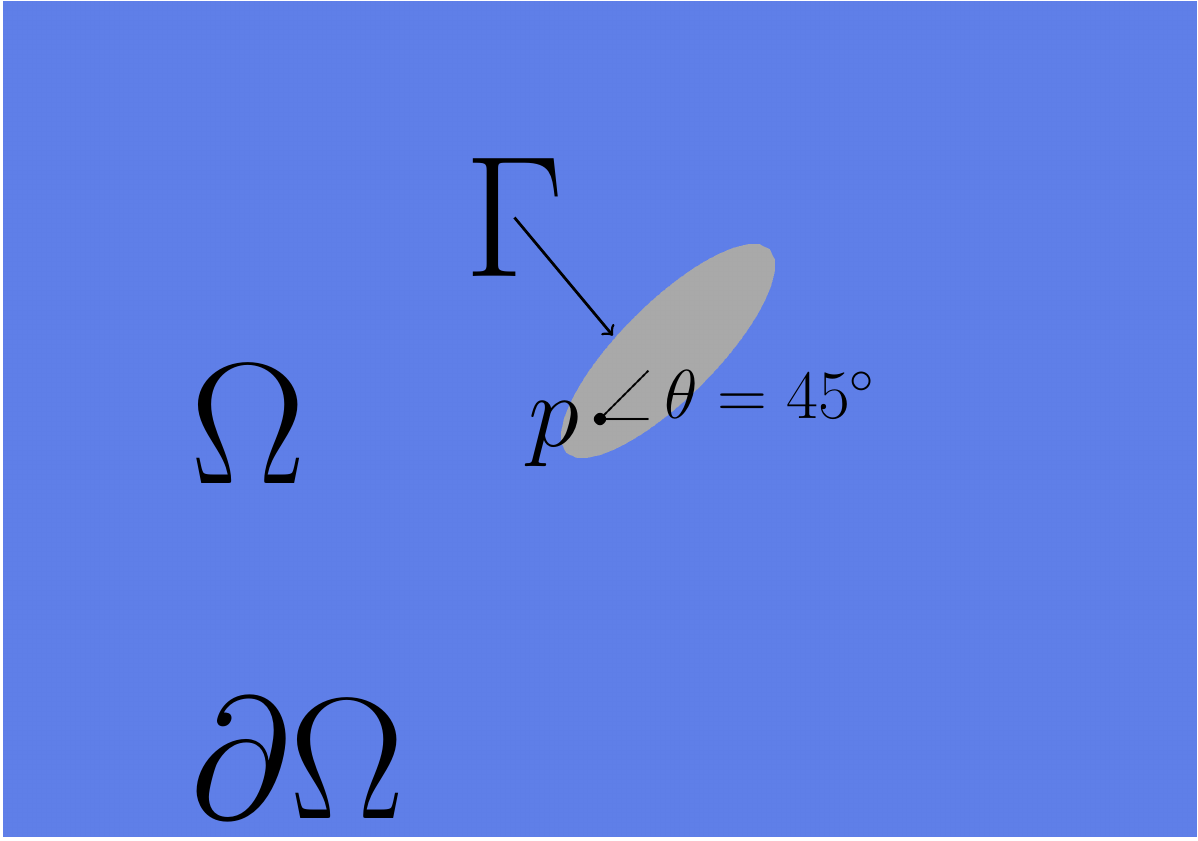}
  \caption{The setup of the example \eqref{eq:SimpleFunctional} and \eqref{eq:SimpleState}. The obstacle, marked in gray, is rotated $45^\circ$ around the point $p$.
  } \label{fig:SinglePoisson}
\end{figure}

 
After differentiation, we discretize the shape optimization problem (\ref{eq:GeneralFunctional}--\ref{eq:GeneralState}) on multiple domains.
For that, the domain $\Omega$ is divided into multiple, possibly
overlapping subdomains $\hat{\Omega}_i, i=0,\dots,N$. 
Further, we restrict the PDE operator $E$ and the state variable $u$ onto these subdomains. 
We denote these restrictions as  $E_{\hat\Omega_i}$ and $u_i$.
The optimization problem (\ref{eq:GeneralFunctional}-\ref{eq:GeneralState}) can now be reformulated to:
\begin{align}\label{eq:overlapping_general_shape_functional}
  \min_{u, \hat \Omega_0, \hat \Omega_1, \dots, \hat \Omega_N} J(u, \hat\Omega_0,\dots,\hat\Omega_N) 
\end{align}
subject to
\begin{align}\label{eq:overlapping_general_state}
  \begin{split}
    E_{\hat\Omega_i}(u_i,\hat\Omega_i)&=0, \quad i=0, \dots, N,\\
    E_{\Lambda_j}(u_{1},\dots, u_j,\hat\Omega_{1},\dots, \hat\Omega_j)
    &= 0 , \quad j = 1, \dots, N,
  \end{split}
\end{align}
where $E_{\Lambda_j}$ are interface conditions that arise from the MultiMesh FEM (see \cref{sec:MMFEM}), which ensures equivalency of \cref{eq:GeneralState} and \cref{eq:overlapping_general_state}.

\examplecont{
  A domain composition of the example problem is visualized in \cref{fig:MultiPreDomain}.}
\begin{figure}[!ht]
  \centering
  \includegraphics[width=0.8\linewidth]{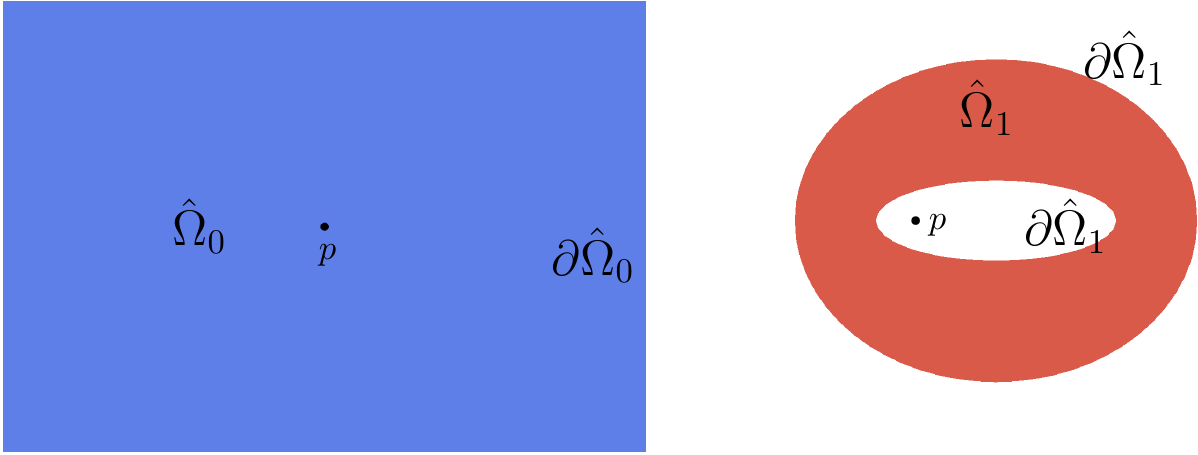}
  \caption{ Subdomains of the multiple domain formulation for example
    (\ref{eq:SimpleFunctional}--\ref{eq:SimpleState}).
    The domain $\hat\Omega_1$ and the obstacle, marked in gray, can be rotated
    around $p$.} \label{fig:MultiPreDomain}
\end{figure}

Formulation (\ref{eq:overlapping_general_shape_functional}--\ref{eq:overlapping_general_state}) is very convenient for discretizing the shape derivative of  (\ref{eq:SimpleFunctional}--\ref{eq:SimpleState}), since each domain can be updated separately.
\FloatBarrier
The formulation (\ref{eq:overlapping_general_shape_functional}--\ref{eq:overlapping_general_state}) results in the following solution algorithm:
\begin{algorithm}[ht]
  \begin{algorithmic}
  \STATE{Set iteration counter $k=0$}
  \STATE{Choose subdomains $\hat \Omega^k_i, i=0,\dots,N$}
  \WHILE{not converged}
  \STATE{Solve the state equations \eqref{eq:overlapping_general_state} on $\cup_{i=0}^N\hat\Omega^k_i$ with
    MultiMesh FEM (\cref{sec:MMFEM})}
  \STATE{Compute the shape sensitivities of functional \eqref{eq:overlapping_general_shape_functional} (\cref{sec:shape})}
  \STATE{Perform optimization step to obtain subdomains
    $\hat \Omega^{k+1}_i, i=0, \dots, N$ (\cref{sec:updatemesh})}
  \STATE{Set $k\gets k+1$}
  \ENDWHILE
  \RETURN{Optimized domain $\bigcup_{i=0}^N\hat \Omega^k_i$}
  \caption{Algorithm for shape optimization with multiple domains.
    \label{alg:shape}}
  \end{algorithmic}
\end{algorithm}

%% file: mmfem.tex
\section{The MultiMesh Finite Element Method}\label{sec:MMFEM}


In this section, we will discuss how to solve the state equation \eqref{eq:GeneralState} with a domain consisting of multiple overlapping sub-domains using the MultiMesh FEM, for which we recall some notation from~\cite{MMFEMarbit}.

\subsection{Domains and Meshes}\label{sc:mm_domains}
In \cref{sec:math}, we introduced a composition of $\Omega$, such that
$\Omega\subseteq\bigcup_{i=0}^N\hat{\Omega}_i$, where $\hat\Omega_i$ is defined
as the \textit{predomain}.
If a point  $x\in\Omega$ can be found in multiple predomains, we associate it with the highest index $i$. Thus, if interpreted visually, the predomain with the higher index appears to be \emph{on top} of the predomain with the lower index.
Since the predomains will overlap, we define the \textit{visible part} of each predomain as $\Omega_i:=\hat\Omega_i\setminus\bigcup_{j=i+1}^N\hat\Omega_j$, $i=0,\dots,N$. Note that $\hat\Omega_N = \Omega_N$. The \textit{visible boundary} of each predomain $\hat\Omega_i$ is denoted $\Lambda_i:=\partial\hat\Omega_i\setminus\bigcup_{j=i+1}^N\hat\Omega_j$, $i=1,\dots,N$.

We define a \textit{premesh} $\hat{\mathcal{K}}_{h,i}$ as the mesh of the
predomain $\hat\Omega_i$, and denote its maximum cell diameter $h_i$.
The elements of $\hat{\mathcal{K}}_{h,i}$ can be divided into the following
three distinct categories: \textit{uncut}, \textit{cut} and
\textit{covered} elements. Uncut elements are the fully visible elements, cut elements are the partially visible elements, and covered elements are the hidden elements. By manually changing the status of elements to covered, topological changes such as holes may be modeled. This is illustrated in the next example.

The $i$-th \textit{active mesh} $\mathcal{K}_{h,i}$ consists of all cut and uncut elements of $\hat{\mathcal{K}}_{h,i}$. We define the cut domain $\Omega_{i}^{cut}$ as the union of all cut elements. Note that $\Omega_N^{cut} = \emptyset$.
The $i$-th \textit{overlap} is defined as $\mathcal{O}_i=\Omega_{i}^{cut}\setminus\Omega_i$, $i=0,\dots,N-1$. This is the hidden part of the active mesh.

\examplecont{
  The predomains $\hat\Omega_0$ and $\hat\Omega_1$ are shown in \cref{fig:MultiPreDomain}. In \cref{fig:MultiActiveDomain} the visible part of each domain, that is, $\Omega_0$ and $\Omega_1$, is illustrated. In \cref{fig:CellTypes}a), the premeshes $\hat{\mathcal{K}}_{h,0}$ and $\hat{\mathcal{K}}_{h,1}$ are illustrated in black and red. \cref{fig:CellTypes}a) also shows the cut, uncut and covered elements. Note that all element in $\mathcal{\hat K}_{h,1}=\mathcal{K}_{h,1}$ are uncut. In \cref{fig:CellTypes}b), a hole has been introduced in the domain by setting all elements in $\mathcal{\hat K}_{h,0}$ that are cut or covered by the obstacle on $\hat\Omega_0$ to being covered. This creates the effect of a hole in $\hat\Omega_0$, since the covered elements will be ignored in the weak formulation. The boundary of the obstacle now becomes a physical boundary, $\Gamma:=\partial\Omega_1\setminus\Lambda_1$. This will be discussed in the next section. The strong form of the state
  equations, \cref{eq:SimpleState} can be written as

\begin{align}\begin{split}\label{eq:SimpleMMStrong}
    E_{\Omega_0}(T_0,\Omega_0)&= -\Delta T_0 - f = 0\quad\text{ in } \Omega_0,\qquad
    T_0 = g, \quad\text{ on } \partial\Omega_0,\\
    E_{\Omega_1}(T_1,\Omega_1)&= -\Delta T_1 - f = 0 \quad\text{ in } \Omega_1,\qquad T_1=g,\quad\text{ on }\Gamma,
  \end{split}\\
  \begin{split}\label{eq:SimpleMMStrongInterface}
    E_{\Lambda_1}(T_0,T_1,\Omega_0,\Omega_1)&=
    \begin{cases}
      \llbracket T\rrbracket = 0 \hspace{9.2ex}\text{ on } \Lambda_1,\\
      n_1\cdot\left\llbracket\nabla T\right\rrbracket =0
      \hspace{3.2ex}\text{ on } \Lambda_1,
    \end{cases}
  \end{split}
\end{align}
where $\cdot$ is the vector dot product, $\llbracket\psi\rrbracket=\psi_1-\psi_0$ denotes the jump.
The normal vector $n_1$  is defined to be pointing outwards of
the domain $\hat\Omega_1$. The two interface conditions on $\Lambda_1$ ensure
sufficient smoothness of the solution $T$ across the $\Lambda_1$.
  }

\begin{figure}[!ht]
  \centering
  \includegraphics[width=0.5\linewidth]{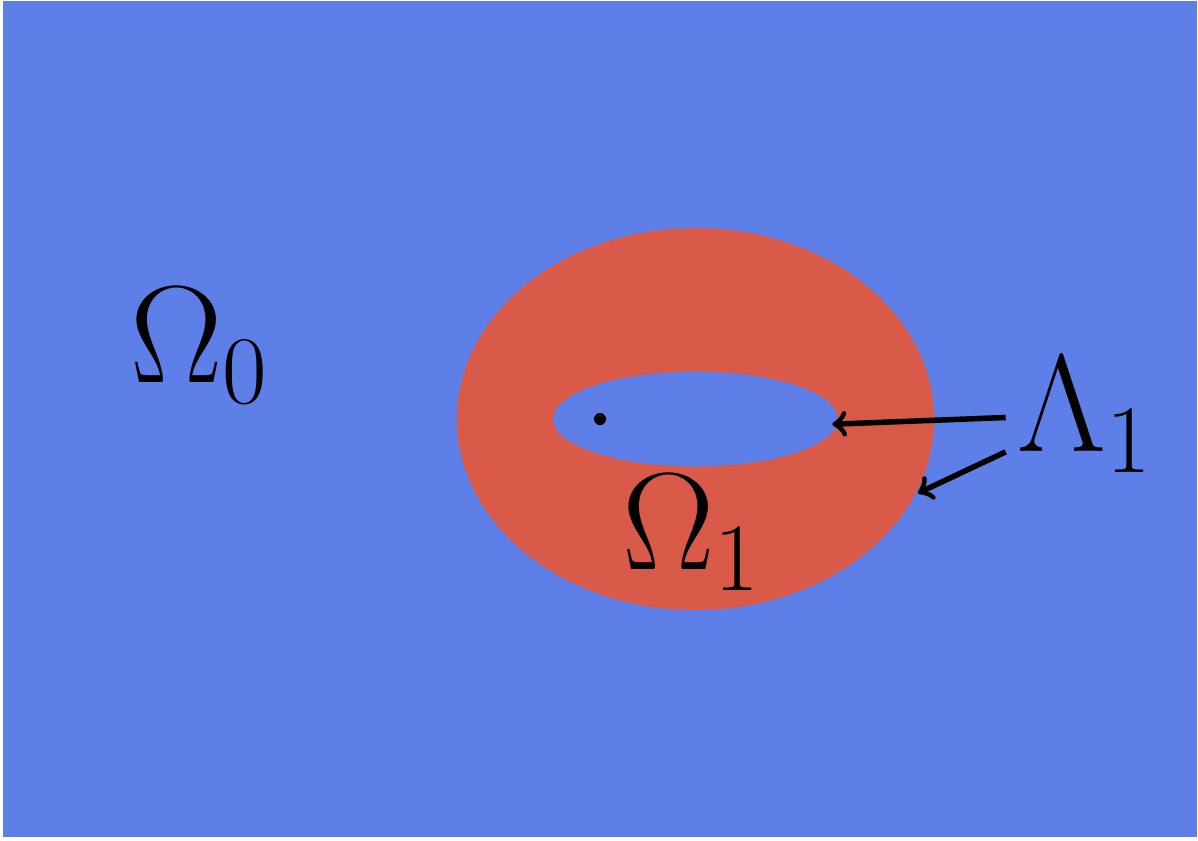}
  \caption{The visual part, $\Omega_0, \Omega_1$ of each predomain $\hat\Omega_0,\hat\Omega_1$ (Shown in \cref{fig:MultiPreDomain}). 
  } \label{fig:MultiActiveDomain}
\end{figure}
\begin{figure}
  \centering
  \begin{tikzpicture}
    \node (sm) [above right, inner sep=0pt]{
      \includegraphics[width=\linewidth]{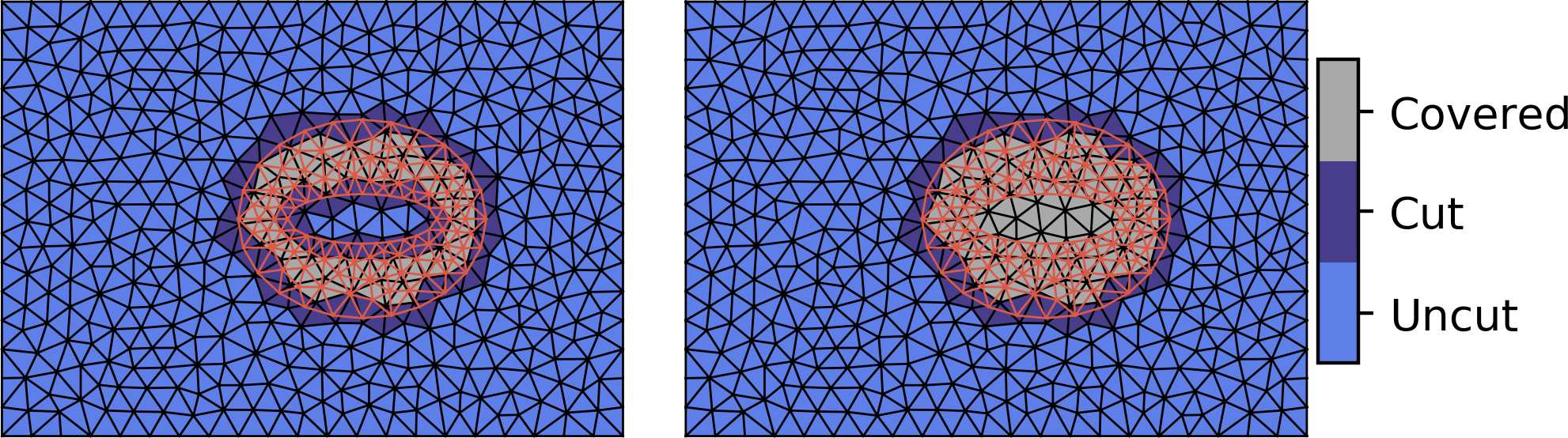}};     
      \node (A) at (0.25\linewidth,-0.25) {(a)};
      \node (B) at (0.66\linewidth,-0.25) {(b)};
  \end{tikzpicture}
  \caption{(a) Visualization of the premeshes $\hat{\mathcal{K}}_0$ (black) and $\hat{\mathcal{K}}_1$ (red). The uncut, cut and covered elements of $\hat{\mathcal{K}}_0$ is shown. (b) The element types after introducing a hole in the domain.}
    \label{fig:CellTypes}
\end{figure}
\subsection{Function Spaces and the Finite Element Method}
 For the weak formulation of \cref{eq:SimpleMMStrong,eq:SimpleMMStrongInterface} the interface conditions are enforced weakly using a Nitsche method~\cite{nitsche1971}. The method contains interior penalty terms, very similar to a discontinuous Galerkin method~\cite{arnold2002unified}, as well as additional stabilization to obtain a stable method. The proposed method is symmetric, stable and yields optimal convergence rates and optimal condition numbers, also in the case of small overlaps~\cite{MMFEMarbit}. Since the interface is not aligned with the meshes, custom quadrature rules are needed to perform the volume and interface integrals that appear in the formulation. See~\cite{Johansson:2017ab} for details.
Let $V_{h, i}$, $i=0,\dots, N$, denote the finite element space on
the active mesh $\mathcal{K}_{h,i}$ consisting of continuous piece-wise
Lagrange polynomials. We define $V_h := \bigoplus_{i=0}^N V_{h, i}$.

\examplecont{
  Let $V_h$ be as described above using $N=1$, and let $V_h^g$ denote the corresponding function space that satisfy the boundary condition. The MultiMesh finite element formulation of the aforementioned example is the following: Find $T = (T_0, T_1) \in V_h^g$ such that
\begin{align}
  \label{eq:SimplePoissonWeak}
    a(T,v)
    +a_{IP}(T,v)+a_O(T,v) = l(v) \quad \forall v = (v_{0}, v_{1}) \in V_h^0,
\end{align}
with
\begin{align}
  a(T,v)&:=\sum_{i=0}^1\Int{\Omega_i}{}\nabla T_i \cdot \nabla v_{i}\md x,\quad
  l(v) := \sum_{i=0}^1\Int{\Omega_i}{}f v_{i}\md x.
\end{align}
The symmetric interior penalty terms are
\begin{align}
  a_{IP}(T,v) &:= \Int{\Lambda_1}{}
  -\langle n_1\cdot \nabla T\rangle \llbracket v\rrbracket
-\llbracket T \rrbracket\langle n_1\cdot\nabla v\rangle +
\frac{\beta_0}{h}\llbracket T\rrbracket\llbracket v\rrbracket\md S,
\end{align}
where $\langle n_1\cdot \nabla \psi\rangle= \frac{1}{2}n_1\cdot(\nabla \psi_1+\nabla \psi_0)$ is the average, $\beta_0$ is a sufficiently large penalty parameter and $h=\frac{h_0+h_1}{2}$.
The overlap stability term is
\begin{align}
    a_O(T,v)&:= \Int{\mathcal{O}_1}{}\beta_1 \llbracket \nabla T \rrbracket \cdot \llbracket\nabla v\rrbracket \md x,
\end{align}
}
  where $\beta_1$ is added for controlling the conditioning of the arising linear system. If not otherwise stated, $\beta_0=\beta_1=4$ is used.


%% file: optimization.tex
\section{Shape Calculus}\label{sec:shape}

The goal of the section is to derive the
(shape) derivative of the objective functional \eqref{eq:overlapping_general_shape_functional} with respect to the domain $\Omega$.
In principal, the shape derivatives can be derived before or after the MultiMesh FEM discretization.
We perform the derivation before the discretization (i.e. from \cref{eq:GeneralFunctional,eq:GeneralState}), which has the benefits as discussed before, namely independence of the software and multi-mesh formulation to be used. The downside is that we introduce a discrete inconsistency in the shape gradient. This inconsistency is only affecting performance when employing coarse meshes. On finer meshes, the inconsistency does not affect performance, as shown later in this section. The chosen approach is visualized in \cref{fig:OptDisc}.
\begin{figure}[!ht]
  \centering
  \usetikzlibrary{arrows,decorations.markings}
  \begin{tikzpicture}[
     decoration={
       markings,
       mark=at position 1 with {\arrow[scale=4,black]{latex}};}]
    \tikzstyle{box}=[draw, fill=cyan!30]
    \coordinate (Opt) at (-\linewidth,0);
    \coordinate (Adjoint) at (-0.4\linewidth,0);
    \coordinate (MMOpt) at (-\linewidth,-3);
    \coordinate (MMAdjoint) at (-0.4\linewidth,-3);
    \node [box,text width=0.3\linewidth] (Sopt) at (Opt)  {Shape optimization problem (\ref{eq:GeneralFunctional}--\ref{eq:GeneralState})};
    \node [box, text width=0.3\linewidth] (Sadj) at (Adjoint) {Shape sensitivity and adjoint equation};
    \node [box, text width=0.3\linewidth] (Madj) at (MMAdjoint) {Weak form of adjoint equation on multiple meshes};
    \node [box, text width=0.3\linewidth] (Mopt) at (MMOpt) {Weak form of state equation on multiple meshes};
    \draw[decoration={markings,mark=at position 1 with
        {\arrow[scale=2,>=stealth]{>}}},postaction={decorate}] (Sopt.east) --  node[below]{\begin{tabular}{c}Shape Calculus\\ Adjoint approach\end{tabular}} (Sadj.west); 
    \draw[decoration={markings,mark=at position 1 with
        {\arrow[scale=2,>=stealth]{>}}},postaction={decorate}] (Sopt.south) --  node[left]{MultiMesh FEM} (Mopt.north);
    \draw[decoration={markings,mark=at position 1 with
        {\arrow[scale=2,>=stealth]{>}}},postaction={decorate}] (Sadj.south) --  node[left]{MultiMesh FEM} (Madj.north);
  \end{tikzpicture}
  \caption{Schematic illustration of the first optimize, then discretize
    approach used in this paper. Shape calculus is employed on the strong formulation of the
    traditional one domain problem. Thereafter, the state equation, and the corresponding
    adjoint equation are discretized using the MultiMesh FEM.
  }\label{fig:OptDisc}
\end{figure}
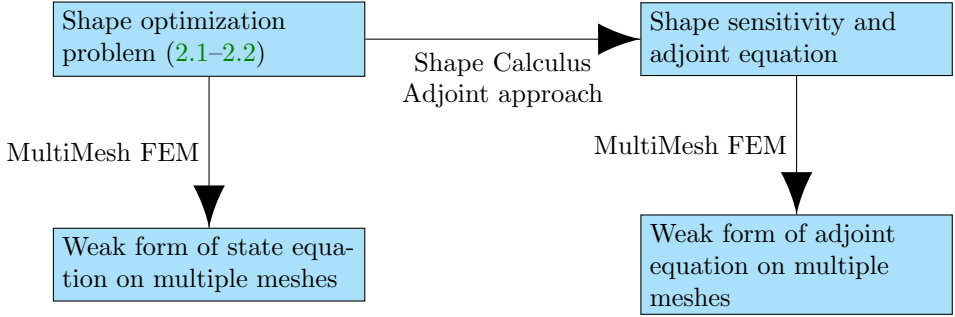

We assume that the state \cref{eq:GeneralState} yields a unique solution $u$ for any given domain $\Omega$.
Then we define the reduced functional as $\hat J(\Omega)=J(u(\Omega),\Omega)$.
We define a perturbed domain $\Omega(\epsilon)$ as
\begin{align}
  \Omega(\epsilon)[s]:=L_\epsilon[s](\Omega)=\left\{ L_\epsilon[s](x):x\in\Omega\right\}\label{eq:omegas}, 
\end{align}
where  $L_\epsilon[s](x):=x(\epsilon):=x+\epsilon s(x)$, $s(x):\Rb^n\to\Rb^n, \epsilon>0$. The shape derivative is then defined as
\begin{align}
  \md \hat J(\Omega)[s]:=\lim_{\epsilon\to 0^+}\frac{\hat J(\Omega(\epsilon))-\hat J(\Omega)}{\epsilon}.
\end{align}
The solution of the PDE on the perturbed domain $\Omega(\epsilon)$ is denoted $u_\epsilon$.
The material derivative and local derivative of $u_\epsilon$ are defined as
\begin{align}
  \md u[s]&:= \lim_{\epsilon\to0^+}\frac{u(\epsilon,x(\epsilon))-u(0,x(0))}{\epsilon}
  ,\quad u'[s]:=du[s]-Du\cdot s,
\end{align}
where $Du$ is the Jacobian.
Using Hadamard's formula, we find the total shape derivative of the functional.

\subsection{Hadamard's formulas}
We consider the cases where the functional is a volume integral or surface integral.
The surface formulation of the Hadamard formulas is used, as the volume formulation would include  expressions containing the overlap and cut interface. The authors plan to investigate the effects of these additional terms in a later publication.
\begin{theorem}[Hadamard Formula for  Volume Objective Functions]
  \label{HadamardVol}
  For a general volume objective function $k:\Omega\rightarrow\mathbb{R}$
  \begin{align}J(\Omega)=\Int{\Omega}{}k\md x,\end{align}
  the shape derivative is given by
  \begin{align}
    \md J(\Omega)[s]&= \Int{\Gamma}{}s\cdot n k\md S + \Int{\Omega}{}k'[s]\md x.
  \end{align}
\end{theorem}
\begin{proof}
  To facilitate the derivative with respect to the perturbed domain, the
integral is transported to the unperturbed domain, as shown below.
\begin{align}
  \begin{split}
   &\md J(\Omega)[s]=\left.\totder{}{\epsilon}\right\vert_{\epsilon=0}
    \Int{\Omega(\epsilon)[s]}{}k(\cdot,\epsilon)\md x
    =\Int{\Omega}{}\left.\totder{}{\epsilon}\right\vert_{\epsilon}\Big(
    k(L_\epsilon(\cdot),\epsilon)\vert \mathrm{det}DL_\epsilon(\cdot)\vert\Big)\md  x\\
    &=\Int{\Omega}{}k\mathrm{div}(s)+\md k[s]\md x=
    \Int{\Omega}{}\mathrm{div}(ks)+k'[s]  \md x\\
    &= \Int{\Gamma}{}s\cdot n k\md S
    +\Int{\Omega}{}k'[s]\md x.
  \end{split}
\end{align}
More details can be found in~\cite{delfour2011shapes}.
\end{proof}
\begin{theorem}[Hadamard Formula for Surface Objectives]\label{HadamardSur}
  For a general surface objective function $h:T(\Gamma)\rightarrow\mathbb{R}$, which is dependent
  of the shape and for which $\der{h}{n}$ exists, the shape derivative for the surface objective
  \begin{align}J(\Omega)=\Int{\Gamma}{} h\md S\end{align}
  is given by
  \begin{align}\md J(\Omega)[s]=\Int{\Gamma}{}s\cdot n\left(\der{h}{n}+\kappa h\right)\md S + \Int{\Gamma}{}h'[s]\md S,\end{align}
  where $\kappa=\text{div}_\Gamma n$ is the tangential divergence of the normal,
  i.e. the additive mean curvature of $\Gamma$.
\end{theorem}
This proof is following the same strategy as \cref{HadamardVol}. Therefore the proof is omitted, and given in~\cite{delfour2011shapes}.

In our case, $h'[s]$ and $k'[s]$ are the local derivatives of the state
variable $u$ with respect to the design parameters.  When discretized, this is
a dense matrix which is prohibited to compute. Instead, we use the adjoint approach to avoid explicit computations of these terms.

\subsection{The Adjoint Approach}\label{sec:adjoint}
We use the Lagrangian approach to obtain the
shape sensitivity and the adjoint equation. 
We start by defining the Lagrangian $\mathcal{L}$ of problem (\ref{eq:GeneralFunctional}--\ref{eq:GeneralState}), assuming that $J$ is a volume integral
\begin{align}
  \mathcal{L}(u,\Omega,\lambda)&=\Int{\Omega}{}j\md x+(\lambda, E(u,\Omega))_\Omega
  ,\label{eq:SimpleLagrangian}
\end{align}
where $\lambda$ is the Lagrange multiplier. Technically, the scalar product used to define the Lagrangian should not depend on the design parameter $\Omega$.
However, as shown above in the proof of \cref{HadamardVol} the shape gradient is expressed on a reference domain.


The directional derivative of the Lagrangian is
\begin{align}
  \md \mathcal{L}[s]&=\Int{\Gamma}{}s\cdot n(j+\lambda E)\md S
  +\Int{\Omega}{} j'[s]+\lambda E'[s]\md x+\Int{\Omega}{}\lambda'[s]E\md x.
\end{align}

The necessary optimality condition states that the directional derivative of the Lagrangian vanishes for all $s$. This yields the following conditions in variational form.
\begin{align}
  \Int{\Gamma}{}s\cdot n(j+\lambda E)\md S&=0 \quad \forall s,\quad \text{(Design Eq.)}\label{eq:Design}\\
  \Int{\Omega}{} \der{j}{u}u'+\lambda \der{E}{u}u'\md x&=0\quad \forall u'\quad \text{(Adjoint Eq.)}\label{eq:Adjoint},\\
  \Int{\Omega}{}\lambda'E\md x&=0\quad \forall \lambda'\quad \text{(State Eq.)}\label{eq:State}.
\end{align}

\examplecont{
For the example problem (\ref{eq:SimpleFunctional}--\ref{eq:SimpleState}), the Lagrangian is
\begin{align}
  \mathcal{L}(T,\Omega,(\lambda^{\Omega},\lambda^{\dOmega}))&= \Int{\Omega}{}T^2\md x + \Int{\Omega}{} \lambda^\Omega(-\Delta T-f)\md x + \Int{\dOmega
  }{}\lambda^{\dOmega} (T-g)\md S\label{eq:ExampleLagrangian}
\end{align}
The adjoint equation\cref{eq:Adjoint} is
\begin{align}
  \Int{\Omega}{}2TT'\md x + \Int{\Omega}{} -\lambda^\Omega\Delta T'\md x + \Int{\dOmega
  }{}\lambda^{\dOmega}T'\md S=0 \quad \forall T'.
\end{align}
Integrating by parts twice, assuming sufficient regularity of $\lambda$, and interpreting the result as the strong formulation yields
\begin{align}
  \begin{split}\label{eq:StrongAdjoint}
    -\Delta\lambda^\Omega&=-2T \quad\text{ in } \Omega,\\
    \lambda^\Omega&=0\quad\text{ on } \partial \Omega,\\
    \lambda^{\dOmega}&=-\der{\lambda^\Omega}{n}\quad\text{ on } \dOmega.
  \end{split}
\end{align}
We observe that the adjoint equation \eqref{eq:StrongAdjoint} is a Poisson
problem with a homogeneous Dirichlet condition on $\dOmega$.
Using the notation of \cref{eq:SimplePoissonWeak}, we obtain the weak formulation on two meshes of the adjoint equation as:
Find $\lambda=(\lambda_0,\lambda_1) \in( V_{0}^0, V_{1}^0)=V$ such that:
\begin{align}
  \begin{split}
  &a(\lambda,v)+a_{IP}(\lambda,v)+a_O(\lambda,v)=-2\sum_{i=0}^1\Int{\Omega_i}{}T_i v_i\md x, \quad \forall v=(v_0,v_1)\in V.
  \end{split}
\end{align}
The design equation~\cref{eq:Design} is
\begin{align}
  \Int{\Gamma}{}s\cdot n\Big( T^2+\lambda^\Omega(-\Delta T-f)+\der{(\lambda^{\dOmega}(T-g))}{n}+\kappa\lambda^{\dOmega}(T-g)\Big)
  \md S=0 \quad\forall s.\label{eq:ExDesign}
\end{align}
Assuming that state~\cref{eq:SimpleState} and adjoint~\cref{eq:StrongAdjoint} equations are satisfied, the left hand side of the design equation corresponds to the gradient of the reduced functional, i.e.
\begin{align}
\md \hat J(\Omega)[s]&=
   \Int{\Gamma}{}s\cdot n\left(T^2-\der{\lambda^\Omega}{n}\der{(T-g)}{n}\right)\md S.
\end{align}
}

\section{Optimization Algorithm and Mesh Deformation}\label{sec:updatemesh}

We use the steepest descent method as the optimization algorithm.
Combined with a line search, the steepest descent method guarantees a decrease in the goal functional by iteratively moving in the opposite direction as the gradient of the functional. More precisely, for a functional $\hat J$, with
design-parameter $\Omega^k$ at iterate $k$, the $k$-th iteration of the steepest
descent method is:
\begin{align}\label{eq:steepest_descent}
  \Omega^{k+1}&=\Omega^k(\xi)[d] 
\end{align}
where $\xi>0$ is the step-length decided by e.g. an Armijo linesearch~\cite{armijo1966minimization}, and  $d: \Omega^k \to \mathbb R^n$ is the Riesz representer of $-\mathrm{d}\hat J$.
It is convenient to incorporate the mesh deformation scheme at this point, as it ensures continuity into the interior and smoothness of the boundary.
This can be achieved for instance by choosing a Riesz representer
in a scaled $H^1$-norm. This results in the following variational problem: find $d$ such that
\begin{align}
\Int{\Omega^k}{}{}\alpha\nabla d\cdot \nabla s+ d\cdot s \md x = -\mathrm{d}\hat J[s] \quad\forall s,\label{eq:OmUp}
\end{align}
where $\alpha\ge0$ can be thought of as a smoothing parameter.

By choosing appropriate test and trial function spaces for \cref{eq:OmUp}, we can control which type of shape deformations are allowed.
For instance, if the design variable is the position of an obstacle, it is natural to create a submesh that contains the obstacle, and to choose the test and trial functions to constant unit vectors restricted to that submesh. The consequence is that the Riesz-representer $d$ is a spatially constant function on the submesh, and the submesh is translated in its entirety, see \cref{fig:Ball}b). 
This is in contrast to a traditional one mesh approach, where the test and
trial function spaces are linear functions per element.
This results in the compression effect, see \cref{fig:Ball}a), and can lead to invalid meshes for larger deformations.
Similarly, if the design variable is the rotation of an obstacle contained in a submesh $\hat \Omega^k_i \subset \mathbb R^2$,
the test and trial space is spanned by a single function describing the rotation velocity of the submesh:
\begin{align}
s(p) =
\begin{cases}
     (-p_y+c_y, p_x-c_x) ,& \text{if } p \in \hat \Omega^k_i, \\
     (0, 0) ,& \text{else},
\end{cases}
\end{align}
where $c = (c_x, c_y)$ is the center of rotation.

In the general case, where the design variables are the node positions on a boundary $\Gamma$ on a submesh $\hat \Omega^k_i$, a natural test and trial space for \cref{eq:OmUp} is the finite element space spanned by continuous, piece-wise linear functions on $\hat \Omega^k_i$. This approach works well for small deformations, but yields degenerated meshes for larger deformations.
A more robust, but computationally more expensive approach is to additionally advect the negative gradient from $\Gamma$ to the other boundaries
\begin{align}
  \begin{split}
\begin{alignedat}{2}
  \Int{\Omega}{}{} \alpha_0\nabla d \cdot \nabla s + d\nabla\epsilon \cdot \nabla s \md x &= 0 & \forall s  \text{ on } \hat \Omega^k_i,\\
  d &= n g(x)&  \quad \text{ on } \Gamma,
  \label{eq:ConvEik}
\end{alignedat}
\end{split}
\end{align}
 where the right hand side stems from the integrand of the shape derivative, i.e. $\md \hat J[t] =\Int{\Gamma}{}t\cdot n g(x)\md\Gamma$, and $\epsilon$ is the solution of a smoothed Eikonal equation:
\begin{align}\begin{split}\label{eq:Eikonal}
  -\alpha_1\Delta\epsilon+\norm{\nabla\epsilon}_2^2&=1 \text{ in } \hat \Omega^k_i,\\
  \epsilon &= 0 \quad\text { on } \Gamma,
\end{split}
\end{align}
with smoothing parameter $\alpha_0,\alpha_1 \ge 0$. The results in this paper have been produced with $\alpha_0=10^{-3}, \alpha_1=25$ unless otherwise stated. Note that the outer boundaries of the submesh $\partial \hat\Omega_i \setminus
\Gamma$ are free to move.
Further note that the deformation equations only have to be solved on the submesh, while the background domain is kept stationary.
 Therefore, the degrees of freedom in the deformation equation are significantly reduced compared to the traditional one mesh approach. 

The other alternative for updating the computational domain is to move the boundary of the original mesh, and use a re-mesh algorithm based on the new boundary.
This approach is not employed in this article, but the authors note that by employing multiple meshes, smaller meshes can be re-meshed to make large changes in the geometry.

\examplecont{
We illustrate the advantage of the multiple domain approach with the example
problem \cref{eq:SimpleFunctional,eq:SimpleState}. The optimal rotation of the obstacle is shown in \cref{fig:SimpleOpt} b), where the optimization algorithm converged after 5 iterations. Since the top-mesh $\mathcal{K}_{h,1}$ has been rotated as an
entity, the mesh-quality is fully preserved.
In \cref{fig:SimpleOpt} a), we verify the solution
by plotting the functional for all different orientations of the obstacle.
}
\begin{figure}[ht]
  \centering
      \begin{tikzpicture}
        \node [above right, inner sep=0pt]{
          \includegraphics[width=\linewidth]{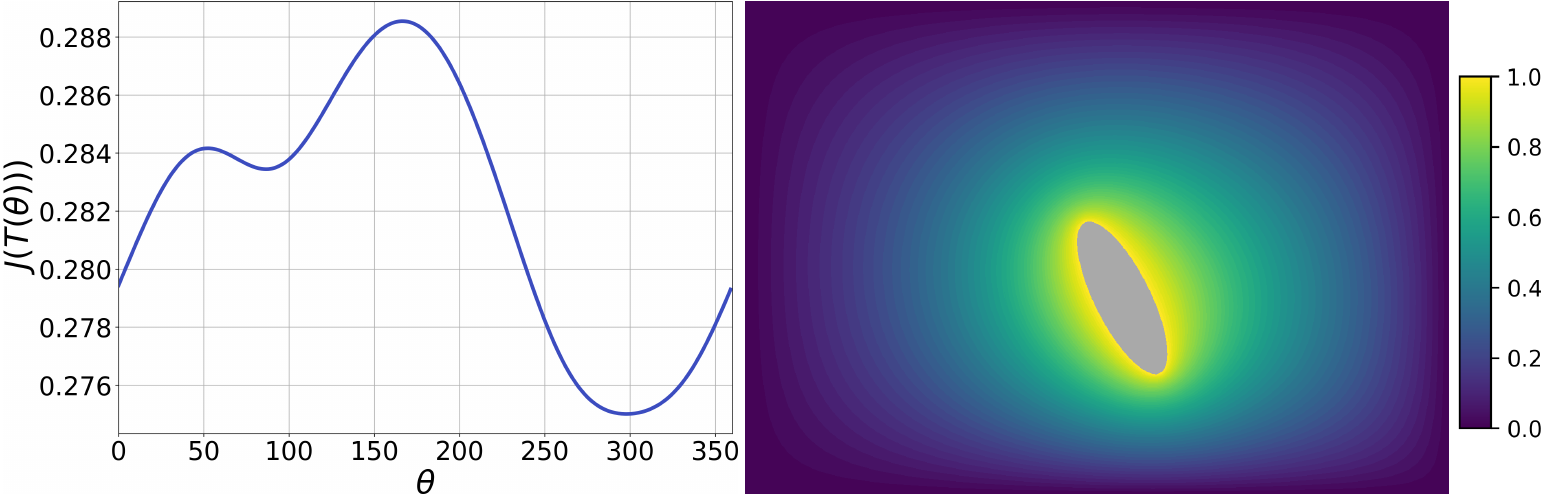}
      };
    \node (B) at (0.25\linewidth,-0.25) {(a)};
    \node (B) at (0.75\linewidth,-0.25) {(b)};
  \end{tikzpicture}
    \caption{(a) The functional $\hat J$ as a function of the rotation angle $\theta$.
      (b) The optimal orientation of the obstacle, with angle
      $\theta=296.6$ where the domain for the initial domain is $\theta=0.0$.
      The optimized orientation was achieved after $5$ steepest-descent iterations, and
      a total of $12$ functional evaluations, including those in the Armijo linesearch.}\label{fig:SimpleOpt}
\end{figure}

%% file: numerical_examples.tex
\section{Numerical examples}\label{sec:numerical_examples}

\subsection{Implementation}
The MultiMesh FEM is implemented in the finite element framework FEniCS~\cite{alnaes2015fenics,logg2010dolfin}, and will be released in version 2018.1.0. All meshes used in this paper have been generated with GMSH, version 3.0.6~\cite{gmsh}. The complete code of the examples are published at \href{https://bitbucket.org/dokken/multimeshshapeopt}{Bitbucket} and access can be granted by emailing the corresponding author.

\subsection{Optimization of Current Carrying Multi-cables}\label{sec:multicable}

In modern cars, the number of electronic devices have been increasing,
especially in electric and hybrid cars. This means that car manufacturers must design 
wires carrying current to the different devices as compactly as possible. An example of such a multi-cable is shown in \cref{fig:multicable_example}.
 This motivates optimizing the design of such multi-cable to minimize the heat inside, see~\cite{harbrecht2016} and the references therein.
As design variables, the authors chose the position of each internal cable of the multi-cable.
Each optimization iteration results in new cable positions and a re-meshing strategy was used to 
update the mesh. Since the internal cables are translated within the multi-cable, the re-meshing step can be avoided if we apply \cref{alg:shape}.

\begin{figure}[!ht]
\centering
\includegraphics[width=0.2\linewidth]{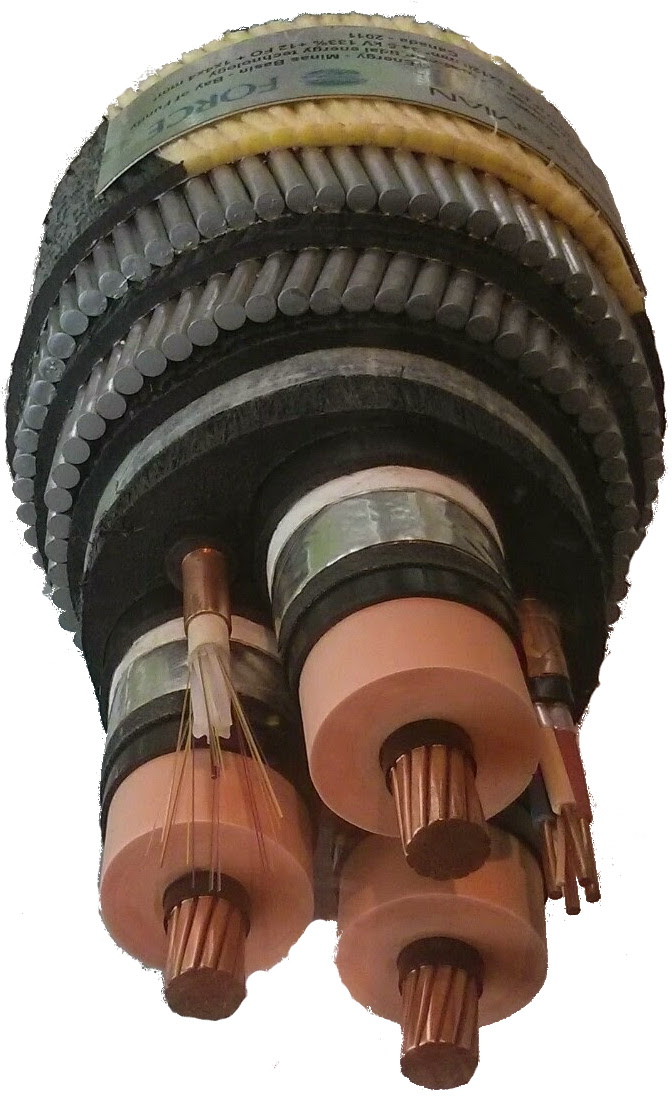}
\caption{A current carrying multi-cable as studied in \cref{sec:multicable}.}
\label{fig:multicable_example}
\end{figure}

To demonstrate this, we consider a simplified multi-cable optimization problem:
\begin{align}  \label{eq:multicable:functional}
  \min_{\Omega, T}J(\Omega,T)=\int_{\Omega}\frac{1}{q}\vert T\vert^q \md x,
  \quad q>1,
\end{align}
subject to
\begin{align}
  \begin{alignedat}{2}
    - \nabla \cdot (\lambda \nabla T) - c T &= f  && \text{ in } \Omega, \\
    \lambda \frac{\partial T}{\partial n} + (T-T^{\text{ex}}) &= 0 &&
    \text{ on } \partial\Omega,\\
         \llbracket T \rrbracket_{\Gamma_{i/e}^j} &= 0 && \text{ on } \Gamma_{i}^{j}\cup\Gamma_{e}^{j}, j=1,\dots, N,\\
           \left\llbracket\lambda\der{T}{n}\right\rrbracket_{\Gamma_{i/e}^j} &= 0 && \text{ on } \Gamma_{i}^{j}\cup\Gamma_{e}^{j}, j=1,\dots, N,
    \label{eq:multicable:state_system}
  \end{alignedat}
\end{align}
where $\Omega = \Omega_{fill} \cup \Omega_{insulation} \cup \Omega_{metal}$ describes a 2D slice through the multi-cable with $N$ internal cables, as specified in \cref{fig:MultiCableDomain} a). The operation $\llbracket \cdot \rrbracket_{\Gamma_{i/e}^j}$denotes the jump over the interface $\Gamma^j_i$ and $\Gamma^j_e$. 
The state equation is a Poisson equation with an additional linear source term with temperature coefficient $c=0.04$. This term describes the rise of electrical resistivity for increasing temperatures in conductive material. The external boundary condition is a Robin-condition, related to the air surrounding the cable, with temperature $T^{\text{ex}}=3.2$.
Furthermore, we set $q=3$ to approximate the $L^\infty$ norm, as done by ~\cite{harbrecht2016}.
A detailed derivation of these equations can be found in~\cite{loos2014joule}.
The source-term $f$ and heat-conductivity $\lambda$ are discontinuous, piece-wise constant functions with the following values:

\begin{center}
  \footnotesize
  \begin{tabular}{|c|c|c|c|}
    \hline
      & $\Omega_{fill}$ &  $\Omega_{insulation}$ & $\Omega_{metal}$ \\
      \hline
      $f$ & 0.0 & 0.0 & 50.0 \\ \hline
      $\lambda$ & 0.08 &  0.19 & 40.0\\ \hline
\end{tabular}
\end{center}
\noindent

The state equation is solved with the MultiMesh FEM for arbitrarily
many intersecting meshes, see \cref{sec:MMFEM} and~\cite{MMFEMarbit}.
There are different ways of creating the overlapping meshes for this
problem. We chose to represent the domain by one mesh for the filling material, and $N$ meshes for the inner cables,
see \cref{fig:MultiCableDomain}b).
The meshes for the internal cables include a halo of filling material, which was chosen sufficiently large so that the heat conductivity $\lambda$ is constant over the cells categorized as overlapped. This guarantees that the solution on the overlap area does not add non-physical contributions to the weak formulation.

We choose the position of each inner cable as a design variable, as described in \cref{sec:updatemesh}. 
Since the background cable is fixed, there are additional constraints on the centroids. 
The distance of each cable to origin has to be bounded, such that the cables do not move
outside the cable defined by $\Gamma^{ex}$. To enforce these
constraints, we used a projected Armijo Rule~\cite{hinze2008optimization}.
\begin{figure}[!ht]
  \centering
  \begin{tikzpicture}
    \node [above right, inner sep=0pt]{
      \includegraphics[width=\linewidth]{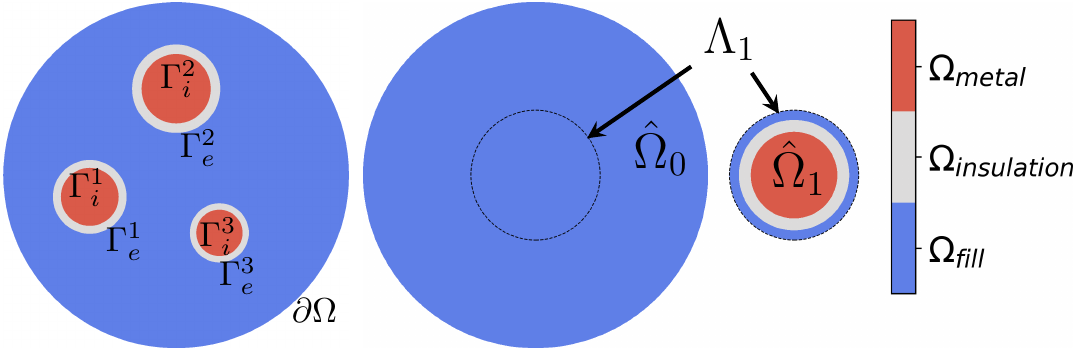}};
    \node (B) at (0.15\linewidth,-0.25) {(a)};
    \node (B) at (0.6\linewidth,-0.25) {(b)};
  \end{tikzpicture}
  \caption{(a) Illustration of the material composition of a multi-cable with annotated boundaries. (b) Illustration of how an internal cable is represented by a separate domain. Every domain includes an extra halo surrounding the cable. A continuity condition is therefore imposed for the state-equations at $\Lambda_1$.}
  \label{fig:MultiCableDomain}
\end{figure}

The adjoint equations are derived following \cref{sec:adjoint} and are:
\begin{align}
  \begin{alignedat}{2}
    -\nabla\cdot(\lambda\nabla p)-cp &= -T\vert T\vert^{q-2}&&
    \text{ in }\Omega,\\
    p&=p_{ex}&& \text{ on } \partial\Omega,\\
    \lambda\der{p}{n}+p_{ex}&=0 &&\text{ on }
    \partial\Omega,\\
    \left\llbracket \lambda\der{p}{n} \right\rrbracket_{\Gamma_{i/e}^j} &=0 && \text{ on } \Gamma_{i}^{j}\cup\Gamma_{e}^{j}, j=1,\dots, N,\\
    \llbracket p  \rrbracket_{\Gamma_{i/e}^j} &=0 && \text{ on } \Gamma_{i}^{j}\cup\Gamma_{e}^{j}, j=1,\dots, N. 
  \end{alignedat}
\end{align}
and the shape sensitivity is
\begin{align}
  \begin{alignedat}{2}
    \md J(\Omega)[s] &= \sum_{j=1}^N \Int{\Gamma_{i}^{j}\cup\Gamma_{e}^{j}}{}s\cdot n\Big(&&
     \llbracket -cTp-fp \rrbracket_{\Gamma_{i/e}^j}
  -\lambda^+\der{p^+}{n}\left\llbracket \der{T}{n} \right\rrbracket_{\Gamma_{i/e}^j}\\ 
&  && +\llbracket \lambda \rrbracket_{\Gamma_{i/e}^j}\nabla_\Gamma p^+\cdot \nabla_\Gamma T^+
  \Big) \md S,
  \end{alignedat}\label{eq:MultiCableGradient}
\end{align}
where the super-script $+$ denotes the evaluation of a function from
the fill side at $\Gamma_{e}^j$, and evaluation at the insulation side of $\Gamma_{i}^j$.

\subsubsection{Results}\label{Multicable:results}
Numerical results were performed on a multi-cable with radius $1.2$, containing $N$ internal cables with $0.2$ radius not counting insulation and a $0.055$ thick insulation. 

The adjoint equation and shape sensitivity were verified by a Taylor test, where one
inner cable without insulation was placed at $(0.03,0.2)$. 
A MultiMesh consisting of a background mesh with $33802$ elements
and $15842$ elements for the mesh of each of the internal cable was used.
The convergence rates in the steepest descent direction are shown in \cref{tab:MultiCable1Taylorsteep}.
Similar convergence rates were obtained for other perturbations,
indicating that the adjoint equation and the shape derivatives are correct.
Further tests with different mesh refinements showed that the Taylor convergence rates
decreases if the mesh is too coarse, due to the discrete inconsistency in the functional sensitivities, see \cref{sec:shape}.

\input{taylor_MC1steep.tex}

Next, the optimization algorithm was verified by optimizing the position of three
identical cables. For this setup, it is known that the optimal positioning of the cables form an equilateral triangle~\cite{harbrecht2016}. 
The optimization loop was terminated when the relative functional reduction dropped below a set tolerance, i.e.
$\left\vert\frac{J(\Omega^{k+1})-J(\Omega^k))}{J(\Omega^{k+1})}\right\vert<10^{-6}$. The optimal configuration was found after $62$ iterations, when the functional decreased from $3.2\cdot 10^4$ to $3.9\cdot 10^3$. The optimized angles between the cables were $58.7, 62.1, 59.3$ degrees. The initial and optimized configurations are shown in \cref{fig:MultiCable3}.

Furthermore, we compared the computational expense of the MultiMesh FEM with a traditional FEM approach.
For this comparison, we considered a problem with one internal cable. We measured the run-time of one approximate iteration, consisting of assembling and solving the state and adjoint systems, and updating the mesh.
For the MultiMesh-approach, the mesh updated consists of translating the mesh coordinates of the inner cable and to recompute the new mesh intersections.

For the traditional one mesh approach, the mesh update was performed through remeshing.
The linear systems arising in both approaches were solved using the an LU solver.
The timing results are shown in \cref{tab:Timing_Poisson}.
The results show that the assembly of the MultiMesh-system is more time consuming that the traditional one mesh approach, primarily caused by the the additional stabilization terms. However, this additional expense is outweighed by a significant lower cost for  the mesh update compared to re-meshing. Therefore, the estimated iteration cost for the MultiMesh-approach is a third of the traditional approach.

Finally, the optimization was performed on a problem with five internal cables of different sizes and with different insulation parameters. The parameters are listed in \cref{tab:MC5}. The initial and optimized cable configurations are shown in \cref{fig:MC5}. The results show that the smallest cable, Cable 5, is placed far away from the other cables. This can be explained through the fact that this cable has the lowest insulation parameter $\lambda_{iso}$ and the largest heat source $f$. The optimal solution was found after $103$ iterations, and the functional decreased from $1.1\cdot 10^5$ to $9.8\cdot 10^3$.

\begin{figure}[!ht]
  \centering
  \begin{tikzpicture}
    \node [above right, inner sep=0pt]{
      \includegraphics[width=\linewidth]{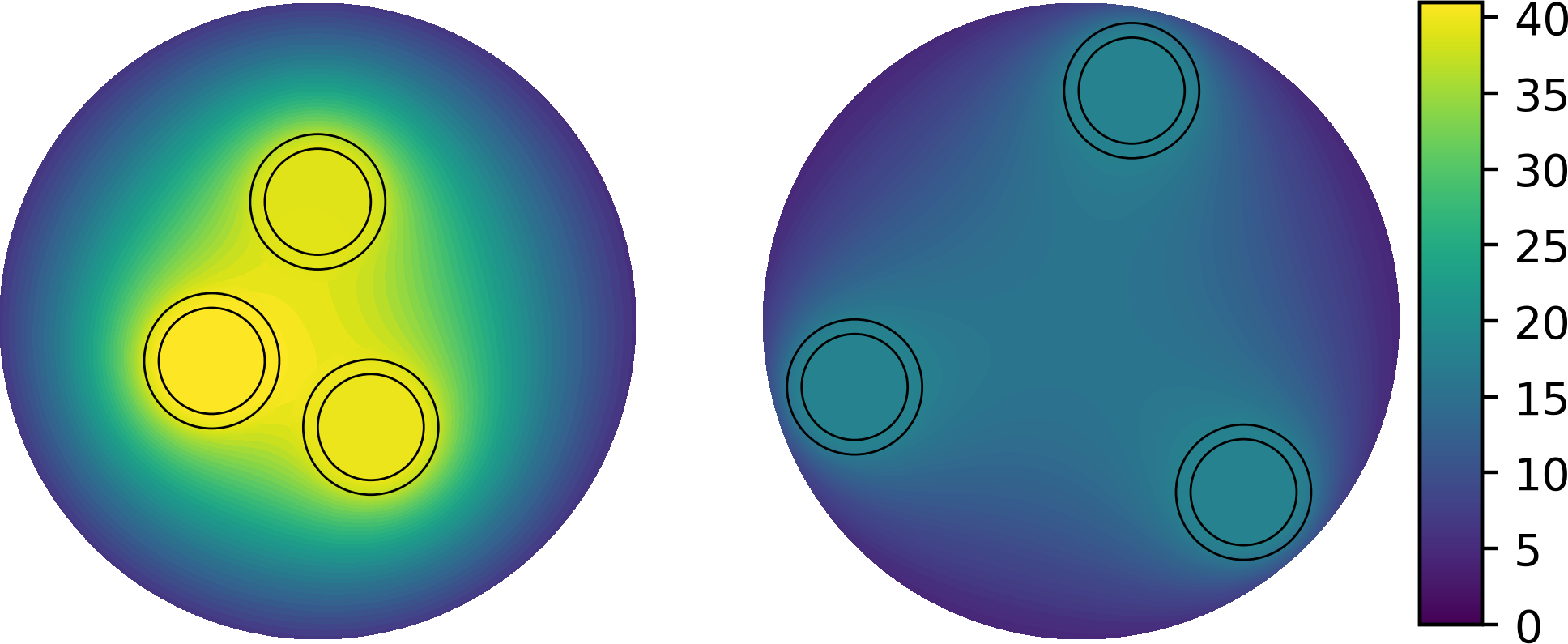}
    };
    \node (B) at (0.2\linewidth,-0.2) {(a)};
    \node (B) at (0.7\linewidth,-0.2) {(b)};
  \end{tikzpicture}
   \caption{
     (a) The cable configuration and temperature distribution the three cables before the optimization.
     (b) The cable configuration and temperature distribution after the optimization. The inner cables from an equilateral triangle.}\label{fig:MultiCable3}
\end{figure}

\input{Timing_Poisson.tex}

\begin{table}[!ht]
  \centering
  \caption{The setup for the 5 multi-cable optimization shown in \cref{fig:MC5}. The parameters $\lambda_{fill},\lambda_{metal}$, are the same as in the other simulations.}\label{tab:MC5}
  \footnotesize
  \begin{tabular}{|c|c|c|c|c|c|}
    \hline   & Cable 1 & Cable 2 & Cable 3 & Cable 4 & Cable 5\\
    \hline Init. Positions & $0, 0.45$ & $-0.4,-0.15$ & $0.2,-0.4$ & $0.5,0.15$ & $0.7,-0.3$  \\
    \hline Opt. Positions & $-0.434, 0.777$ & $-0.903,  0.060$  &
    $-0.191, -0.9$ &  $0.536,  0.784$  & $0.902, -0.384$  \\
    \hline $r_{iso}$ & $0.255$ & $0.242$ & $0.218$ & $0.174$ & $0.122$ \\
    \hline $r_{met}$ & $0.2$ & $0.19$ & $0.171$ & $0.137$ & $0.096$ \\
    \hline $\lambda_{iso}$  & $0.19$ & $0.162$ & $0.133$ & $0.105$ & $0.076$ \\
    \hline $f$ & $50$ & $60$ & $70$ & $80$ & $90$ \\ \hline
  \end{tabular}
\end{table}

\begin{figure}[!ht]
  \centering
  \begin{tikzpicture}
    \node [above right, inner sep=0pt]{
      \includegraphics[width=\linewidth]{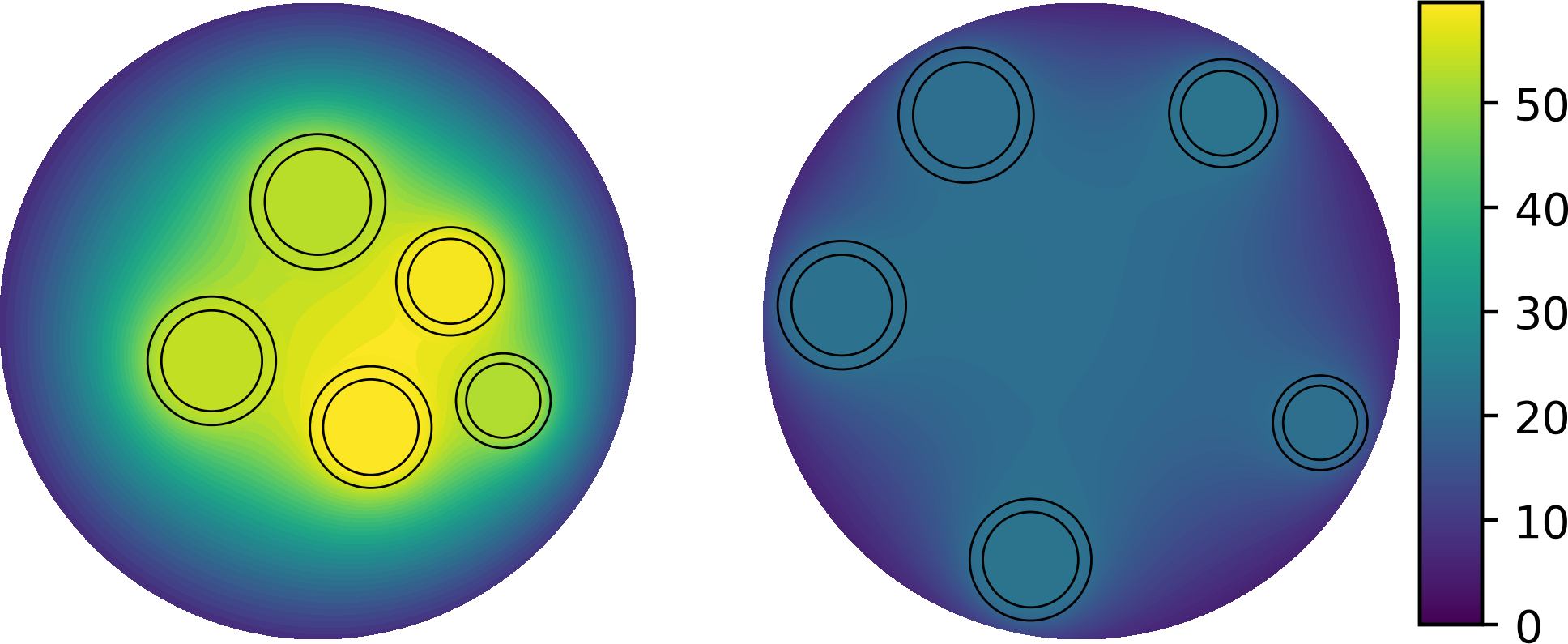}
    };
    \node (B) at (0.2\linewidth,-0.2) {(a)};
    \node (B) at (0.7\linewidth,-0.2) {(b)};
  \end{tikzpicture}
  \caption{(a) Initial configuration of five cables of different sizes and with different material parameters. (b) The optimized positioning of the cables. The smallest cable is placed as far away from the other cables because it has the lowest insulation and highest heat source.} \label{fig:MC5}
\end{figure}
\FloatBarrier

\subsection{Shape Optimization of an Obstacle in Stokes Flow}\label{sec:Stokes}
This example considers the drag minimization of an object
subject to a Stokes flow in two dimensions. This problem has a known analytical solution
first presented in~\cite{pironneau1974optimum}. 
The drag is measured by the dissipation of kinetic energy into heat, that is 
\begin{align}
  J_S &=\Int{\Omega}{} \sum_{i,j=1}^2\left(\der{u_i}{x_j}\right)^2\md x,
\end{align}
where $\der{u_i}{x_j}$ denotes the derivative of the i-th velocity component in
the j-th direction.
The domain consists of a unit square excluding an obstacle as shown in \cref{fig:channel} (a).
\begin{figure}[ht!]
  \centering
  \begin{tikzpicture}
    \node [above right, inner sep=0pt]{
      \includegraphics[width=\linewidth]{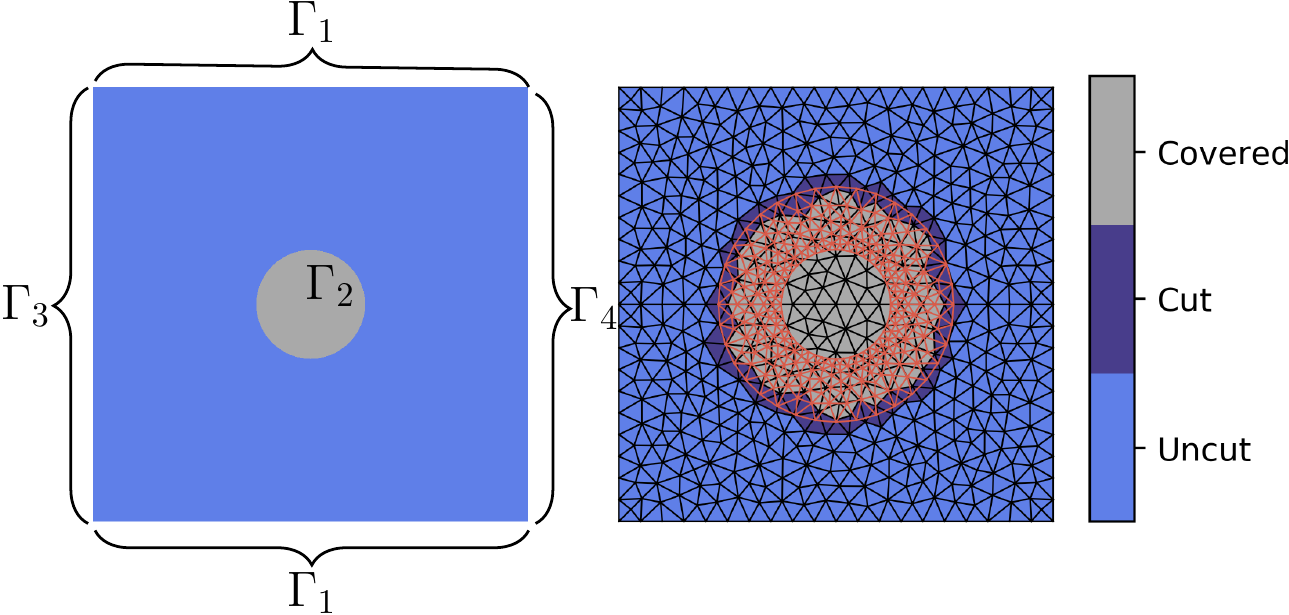}
    };
    \node (B) at (0.25\linewidth,-0.3) {(a)};
    \node (B) at (0.65\linewidth,-0.3) {(b)};
  \end{tikzpicture}
  \caption{(a) The computational domain of shape optimization problem in Stokes flow with
  labeled boundaries. 
  The fluid part of the domain is shown in blue, while the obstacle is shown in gray. 
  (b) A coarse variant of the premeshes visualized by black and red edges. The different cell-types of the background mesh are marked in colors.
  }\label{fig:channel}
\end{figure}
The trivial solution to this problem would be to remove the object from the Stokes-flow completely.
This is avoided by introducing additional constraints on the area and
centroid of the obstacle.
Denoting the target centroid and area of the obstacle as $(C_{x0},C_{y0})=(0.5,0.5)$ and the $V_O =0.05 $ respectively,
we enforce constraints with quadratic penalty terms, yielding the cost functional
\begin{equation}
  \begin{alignedat}{1}
    J_V &= \gamma_1 \left(\vert \Omega \vert - \vert \Omega_0 \vert \right)^2,\\
    J_{Cx} &= \gamma_2 (\text{C}_x-\text{C}_{x0})^2, \\
    J_{Cy} &= \gamma_2 (\text{C}_y-\text{C}_{y0})^2,\\
    J &= J_S + J_V + J_{Cx} + J_{Cy},
    \end{alignedat}
\end{equation} 
with penalty parameters $\gamma_1 > 0$ and $\gamma_2 > 0$. We denote the target fluid area as $\vert\Omega_0\vert=1-V_O=0.95$,
the actual fluid area as $\vert \Omega \vert = \int_\Omega 1 \md x$,  and
the coordinate of the obstacle's centroid as $C_x$ and $C_y$, for instance $C_x =\Big(\half-\Int{\Omega}{}x\md x \Big)
 / {\Big(1- \vert \Omega \vert \Big)}$.

To summarize, we can write the optimization problem as
\begin{align}
  \begin{split}\label{eq:Stokes:functional}
    \min_{\Omega, u} J(\Omega, u) 
  \end{split}
\end{align}
subject to
\begin{equation}
  \label{eq:Stokes:state}
  \begin{alignedat}{2}
    -\Delta  u + \nabla p &= 0 &&\quad\text{ in } \Omega,\\
    \nabla\cdot  u &= 0 &&\quad\text{ in } \Omega,\\
     u &= 0 &&\quad\text{ on }\Gamma_2,\\
     u &= u_0 &&\quad\text{ on } \Gamma_1\cup\Gamma_3,\\
    \der{u}{n} + pn &= 0 &&\quad\text{ on } \Gamma_4,
  \end{alignedat}
\end{equation}
where $p$ is the fluid pressure, $u_0$ a prescribed boundary velocity, and the domain $\Omega$ is a function of $\Gamma_2$. The boundaries $\Gamma_i, i=1,\dots,4$ are visualized in \cref{fig:channel}(a).
The problem was solved using two overlapping domains as visualized in \cref{fig:channel}(b). The annulus describing the front mesh, visualized in \cref{fig:channel}(b) was chosen to have a width of at least three cells of the background mesh. This is important so that one can identity holes added in the geometry by setting cells to be covered (See \cref{fig:CellTypes}). An interesting effect of this choice is that the front mesh scales with the mesh size, as shown in \cref{Stokesconvmesh}. This means that when we refine the mesh, the number of elements that needs to be deformed does not depend on the total number of degrees of freedom of the MultiMesh.
This means that employing the same deformation scheme on the front mesh would be much more efficient than employing it on the whole domain.

We used the variational formulation of the Stokes equations for two overlapping domains, as derived and discretized in~\cite{johansson2015high},
The shape sensitivity of $J_s$ has been derived in~\cite{schmidtphd} and is
\begin{align}\label{eq:stokesgrad}
  \md J_s(\Omega, u,p)[s]&= \Int{\Gamma_2}{}-s\cdot n
  \left(\der{ u}{n}\cdot\der{ u}{n}\right)\md S.
\end{align}
The shape sensitivity of $J_V$ is obtained by applying the product rule and \cref{HadamardVol}:
\begin{align}\label{eq:volumegrad}
  \md J_V(\Omega)[s]=-2\gamma_1(\vert\Omega\vert-\vert\Omega_0\vert)\Int{\Gamma_2}{}s\cdot n\md S.
\end{align}
Similarly, the shape sensitivity of $J_C$ is obtained using the quotient rule:
\begin{align}\label{eq:centroidgrad}
  \md J_{Cx}(\Omega)[s] &= 2\gamma_2\vert\Omega\vert^{-1}\Int{\Gamma_2}{}s\cdot n(C_x-x)(\text{C}_x-\text{C}_{x0})\md S.
\end{align}
Similar result can be derived for $\md J_{Cy}$.
Combining \cref{eq:stokesgrad,eq:volumegrad,eq:centroidgrad} and obtain the
shape sensitivity
\begin{equation}
  \begin{alignedat}{1}
  \md J(\Omega, u,p)[s]&= \Int{\Gamma_2}{}s\cdot n\Bigg(
   -\left(\der{ u}{n}\cdot\der{ u}{n}\right)
  -2\gamma_1(\vert\Omega\vert-\vert\Omega_0\vert)\\
  &+2\gamma_2\inv{\vert\Omega\vert}\Big[(C_x-x)(\text{C}_x-\text{C}_{x0})  +(C_y-y)(\text{C}_y-\text{C}_{y0})\Big]\Bigg)\md S.
  \end{alignedat}\label{eq:stokes_shape_sensitivity}
\end{equation}
We note that \cref{eq:stokes_shape_sensitivity} does not depend on the adjoint solution.
Hence, for this example one does not have to compute the solution of the adjoint equations.

\subsubsection{Results}\label{sec:stokes_res}
Both a traditional FEM and a MultiMesh FEM Stokes solver was implemented with FEniCS~\cite{alnaes2015fenics,logg2010dolfin}. The meshes used for the experiments are shown in~\cref{Stokesconvmesh}.
Both solvers use a Taylor-Hood element pair for the discretization, i.e. second order piece-wise continuous
polynomials for the velocity and first order piece-wise continuous polynomials for the pressure.
The arising linear systems were solved using the direct solver MUMPS~\cite{MUMPS:1}. 
For finer discretizations, it would be beneficial to employ an iterative solver, 
though efficient preconditioning for the MultiMesh variational problem has not been properly explored.
The following results use penalty values of $\beta = 6$ in the variational problem defined in~\cite{johansson2015high}, if not otherwise stated.

First, the convergence rates of the Stokes solvers were verified using the manufactured solution given in~\cite{johansson2015high}. For this setup, we expect third and second order convergence rates for velocity and pressure, respectively.
The results, listed in \cref{manufactured_stokes}, show that both solvers achieve the expected convergence-rates.

\begin{figure}[!ht]
  \centering
    \begin{tikzpicture}
      \node [above right, inner sep=0pt]{
        \includegraphics[width=\linewidth]{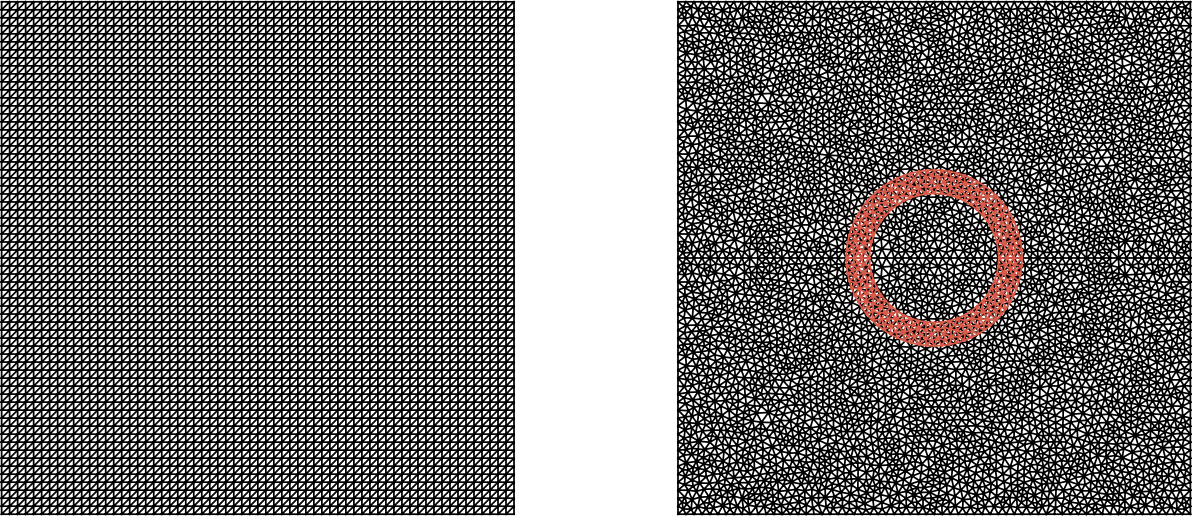}
    };
    \node (B) at (0.25\linewidth,-0.25) {(a)};
    \node (B) at (0.75\linewidth,-0.25) {(b)};
  \end{tikzpicture}
    \caption{a) The traditional FEM mesh and b) the MultiMesh FEM meshes for verification of the
      Stokes solver. For the MultiMesh approach, the front mesh is scaled to
        such that the width of the annulus is equivalent to three background cells. In this setup, no extra cells are marked as covered, and there are two interfaces, as shown in \cref{fig:CellTypes}a).}\label{Stokesconvmesh}
\end{figure}
\input{stokes_manufactured}

Next, the shape sensitivity was verified using a Taylor test, as shown in \cref{tab:stokestaylor}. The expected convergence rate, $2$, is obtained. Further tests in random perturbation directions showed similar convergence rates. 
We also performed Taylor tests on meshes with different resolution, which revealed that the convergence rate reduces on very coarse meshes due to the discrete inconsistency of the shape sensitivity, see~\cref{sec:shape}.

\input{taylor_stokes.tex}

Finally, we solved the full shape optimization problem.
To ensure that the volume and centroid constraints are sufficiently satisfied, 
we solved a sequence of optimization problems, with increasing penalty coefficients $\gamma_1$ and $\gamma_2$.
The optimized mesh of the previous problem is used as an initial mesh for the next optimization problem.
Starting with $\gamma_1=10^5$ and $\gamma_2=10^3$, five optimization problems were solved in which each $\gamma_1$ and $\gamma_2$ were doubled. The number of cut and uncut cells in the initial MultiMesh was 27490 cells.
After a total of 359 optimization iterations, we obtain the final configuration, see~\cref{fig:Stokes90}.
The functional dropped from initially $22.5$ to $18.4$. We observe that the final shape is visually in agreement with~\cite{pironneau1974optimum}, which states that the front and back angle of the object should be $90$ degrees. We measured the front and back angle of our optimized solution to be $87^\circ$ and $84^\circ$ respectively.
The mesh deformation was performed using the advection scheme \cref{eq:ConvEik}. In addition, a centroidal Voronoi tessellation~\cite{du1999centroidal} (CVT) was used on $\Gamma_2$ to preserve the mesh quality near the the front and back wedge. 
For the initial mesh, the minimum cell radius ratio was $0.58$, where an equilateral triangle has measure $1$. For the optimized mesh, this had decreased to $0.32$. 
Compared to a standard FEM, with a mesh with 25420 cells, an iteration of the optimization algorithm is approximately two thirds with the MultiMesh FEM, compared to standard approach. While the assembly of the linear system for the MultiMesh problem is 4 times slower than for the traditional FEM (1.4 seconds to 0.35 seconds), the deformation of the domain is over 20 times faster (0.35 seconds to 7.85 seconds) for the MultiMesh FEM compared to the traditional FEM. In these timings, we have excluded the execution time of the CVT, as they should be equal for two boundaries with similar resolution. The mesh quality of the optimized traditional FEM solution was 0.37.
\begin{figure}[!ht]
  \centering
  \begin{tikzpicture}
      \node [above right, inner sep=0pt]{
  \includegraphics[width=\linewidth]{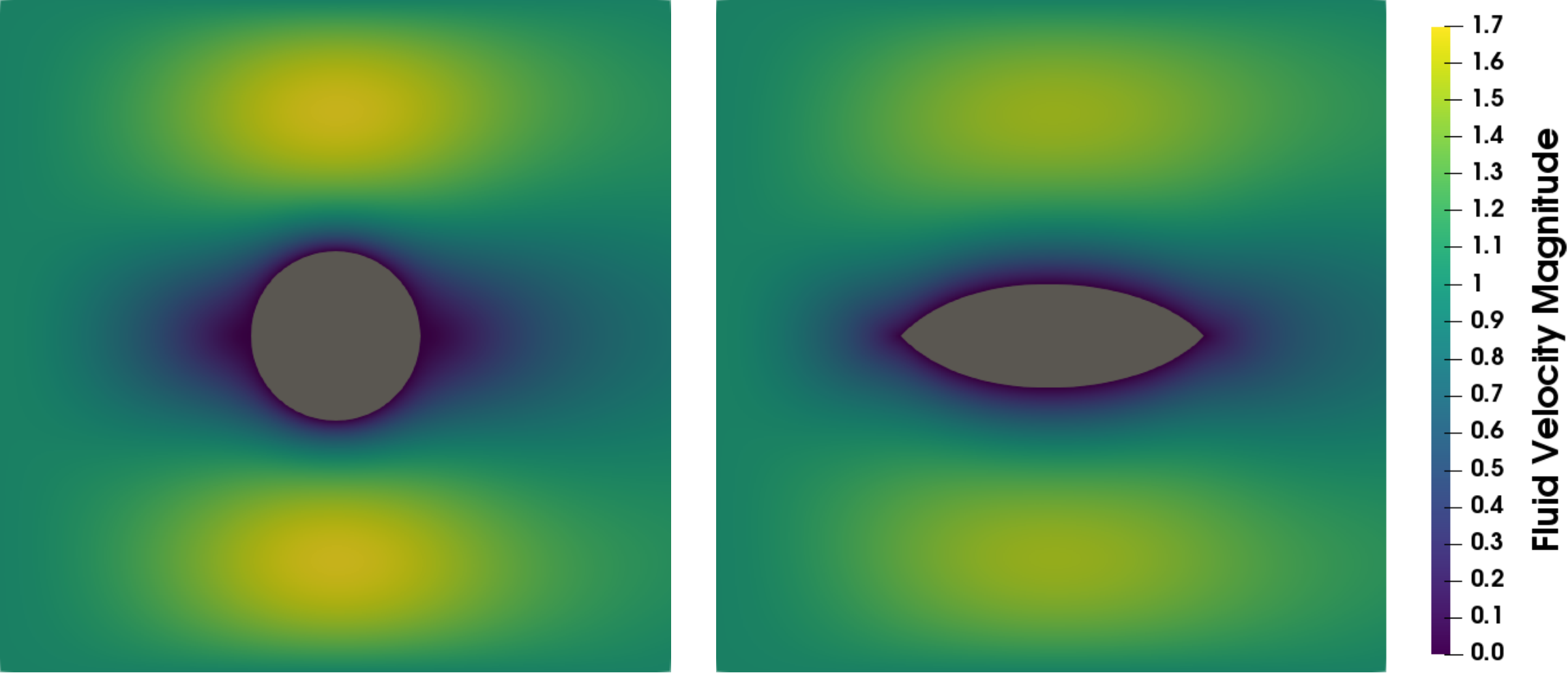}
     };
    \node (B) at (0.25\linewidth,-0.2) {(a)};
    \node (B) at (0.7\linewidth,-0.2) {(b)};
          \end{tikzpicture}
  \caption{The flow around the initial obstacle. (b) Flow-profile of the fluid around the optimized obstacle. A front and back-angle of 90 degrees are obtained, as proven analytically in Pironneau~\cite{pironneau1974optimum}.}\label{fig:Stokes90}
\end{figure}

\subsection{Orientation of 25 objects in Stokes-flow}
As a final example, we considered the problem of optimally rotate 25 obstacles in Stokes flow to minimize dissipation of energy.
We consider 25 identical objects placed in a structured fashion, as shown in \cref{fig:25stokes}a). There are two identical inlets, with parabolic inlet profiles, and one outlet, that is $66\%$ of total inlet width. The optimization was performed using a MultiMesh consisting of a total of 26 meshes, where each obstacle was represented by a separate mesh, and $22,741$ cut and uncut elements. The stopping criteria of the optimization algorithm was set to  ${(J(\Omega^{k+1})-J(\Omega^k))}/{J(\Omega^{k+1})}<10^{-5}$ and achieved after 50 iterations. The optimized configuration is shown in \cref{fig:25stokes}b).

\begin{figure}[!ht]
  \centering
  \begin{tikzpicture}
      \node [above right, inner sep=0pt]{
        \includegraphics[width=\linewidth]{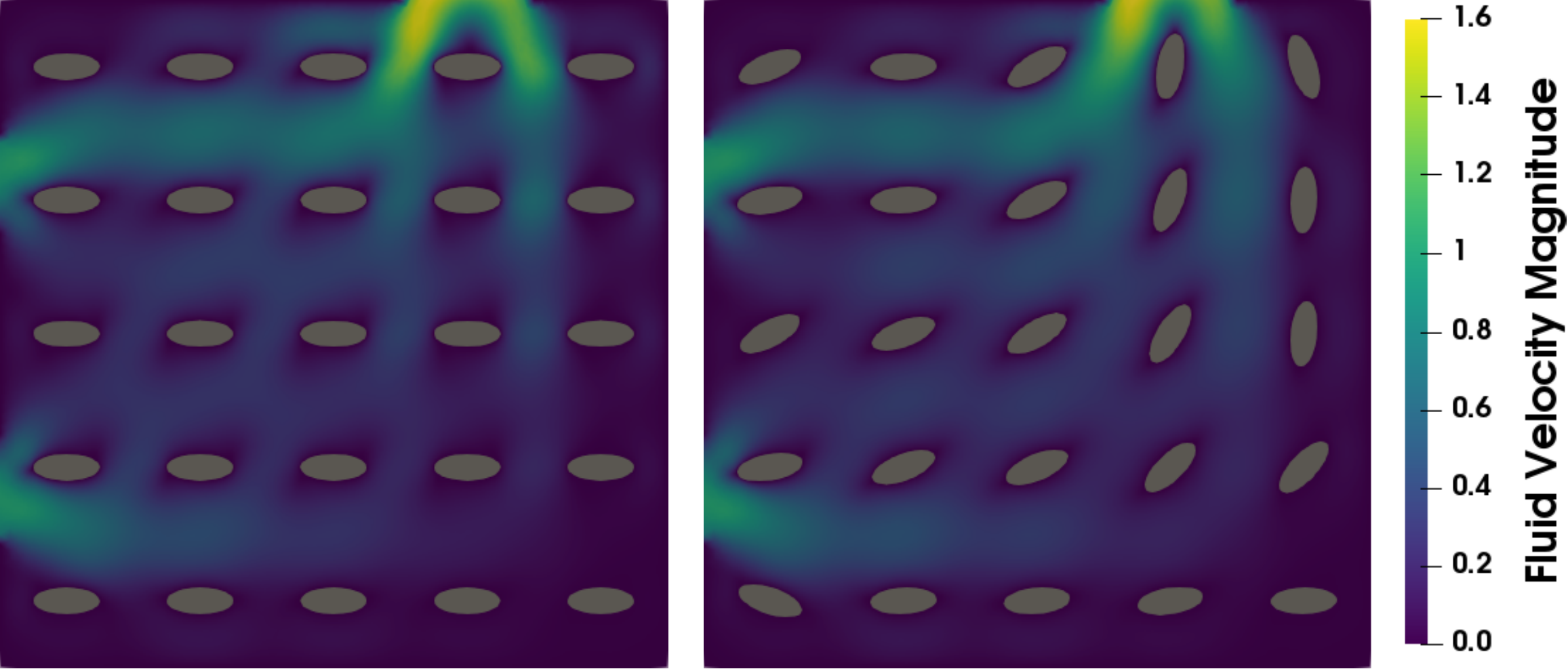}
    };
    \node (B) at (0.25\linewidth,-0.2) {(a)};
    \node (B) at (0.7\linewidth,-0.2) {(b)};
  \end{tikzpicture}
  \footnotesize
  \caption{(a) Initial configuration of 25 objects in a channel with two inlets at the left wall, and one larger outlet at the top. Initial $J_s=72.30$.(b) The optimal orientation of the 25 objects, where each are rotated around their baricenter. Optimal $J_s=56.57$. $50$ iterations where need to obtain the optimized domain. }\label{fig:25stokes}
\end{figure}

%% file: taylor_MC1steep.tex
\begin{table}[!ht]\centering
    \caption{Taylor test for the current carrying multi-cable example, see \cref{Multicable:results}, perturbed in the steepest descent-direction. The first order residual is defined as $R_0=\vert J(\Omega(\epsilon)[d])-J(\Omega)\vert$, and the second order residual as $R_1=\vert J(\Omega(\epsilon)[d])-J(\Omega)-\epsilon\md J(\Omega)[d]\vert$. The expected first and second order convergence are obtained for the first and second order residuals. }\label{tab:MultiCable1Taylorsteep} 
    \footnotesize
    \begin{tabular}{|c|c|c|c|c|} 
      \hline$\epsilon$ & $R_0(\epsilon)$ & order & $R_1(\epsilon)$ & order \\ 
      \hline 1.70e-04 & 2.404e+03 & - &1.089e+03 & -\\ 
      \hline $\epsilon/2$ & 9.743e+02 & 1.30 & 3.168e+02 & 1.78\\ 
      \hline $\epsilon/4$ & 4.112e+02 & 1.24 & 8.247e+01 & 1.94\\ 
      \hline $\epsilon/8$ & 1.847e+02 & 1.15 & 2.035e+01 & 2.02\\ 
      \hline $\epsilon/16$ & 8.697e+01 & 1.09 & 4.783e+00 & 2.09\\ 
      \hline $\epsilon/32$ & 4.213e+01 & 1.05 & 1.034e+00 & 2.21\\ 
      \hline\end{tabular} 
\end{table}

%% file: Timing_Poisson.tex
\begin{table}[!ht]
  \footnotesize
  \caption{The minimum time each operation takes in seconds (5 measures). The remeshing of a single mesh has been done with GMSH 3.0.6~\cite{gmsh}, where we have measured the time it takes to convert the traditional mesh geo-file to a msh file. Note that assembling the MultiMesh variational form takes more time than the traditional FEM system, due to the additional terms in the variational form. A typical iteration in an optimization algorithm (excluding linesearches) includes two assembly and solves (state and adjoint) equation, one mesh update (re-meshing or translation), and a re-computation of intersections for the multimesh (build).}\label{tab:Timing_Poisson}
  \begin{tabular}{|c|c|c|c|c|c|c|} 
    \hline  & Cells  & Assembly & Solve & Mesh Update & Build & App. It. \\ \hline 
    MultiMesh FEM & 47728 & 1.39e-01 & 1.07e-01 & 2.10e-04 & 3.81e-02 & 5.30e-01\\
    \hline Traditional FEM & 46178 & 8.44e-02 & 1.12e-01 & 1.12e+00 & - & 1.51e+00\\
\hline\end{tabular} 
\end{table}

%% file: stokes_manufactured.tex
\begin{table}[!ht]
  \centering
  \caption{Error and convergence rates for the Stokes problem. The expected convergence-rates are achieved.}\label{manufactured_stokes}
  \footnotesize
\begin{tabular}{|c|c|c|c|c|}
\hline \multicolumn{5}{|c|}{\textbf{Traditional FEM}}\\
 \hline Max Mesh size & $L^2$-error in $u$ & Rate $u$ & $L^2$-error in $p$ & Rate $p$\\ 
\hline 0.088 & 1.378e-03 & - & 7.390e-03 & - \\
\hline 0.044 & 1.693e-04 & 3.025 & 1.716e-03 & 2.107 \\
\hline 0.022 & 2.108e-05 & 3.005 & 4.140e-04 & 2.051 \\
\hline 0.011 & 2.635e-06 & 3.000 & 1.018e-04 & 2.024 \\
\hline 0.006 & 3.295e-07 & 2.999 & 2.526e-05 & 2.011 \\
\hline \end{tabular}
\begin{tabular}{|c|c|c|c|c|}
\hline \multicolumn{5}{|c|}{\textbf{MultiMesh FEM}}\\
 \hline Max Mesh size & $L^2$-error in $u$ & Rate $u$ & $L^2$-error in $p$ & Rate $p$\\ 
\hline 0.086 & 8.180e-04 & - & 1.535e-02 & - \\
\hline 0.045 & 1.049e-04 & 3.165 & 2.362e-03 & 2.884 \\
\hline 0.023 & 1.250e-05 & 3.268 & 4.583e-04 & 2.518 \\
\hline 0.012 & 1.560e-06 & 3.076 & 1.061e-04 & 2.162 \\
\hline 0.006 & 1.935e-07 & 3.183 & 2.522e-05 & 2.192 \\
\hline \end{tabular}
\end{table}

%% file: taylor_stokes.tex
\begin{table}\centering
  \caption{Results of the Taylor tests for the deformation of an obstacle in a Stokes-flow. The first and second order residuals are defined as $R_0=\vert J(\Omega(\epsilon)[s])-J(\Omega)\vert,$ $R_1=\vert J(\Omega(\epsilon)[s])-J(\Omega)-\epsilon\md J(\Omega)[s]\vert$.
The table show the results for $s=-(1.2e+4\sin(6\pi x)+1e4\cos(0.1\pi y))n$. 
}\label{tab:stokestaylor} 
\footnotesize
\begin{tabular}{|c|c|c|c|c|} 
\hline$\epsilon$ & $R_0(\epsilon)$ & order & $R_1(\epsilon)$ & order \\ 
\hline 5.000e-06 & 3.230e+02 & - &3.099e+02 & -\\ 
\hline 2.500e-06 & 6.498e+01 & 2.31 & 5.844e+01 & 2.41\\ 
\hline 1.250e-06 & 1.579e+01 & 2.04 & 1.251e+01 & 2.22\\ 
\hline 6.250e-07 & 4.521e+00 & 1.80 & 2.886e+00 & 2.12\\ 
\hline 3.125e-07 & 1.511e+00 & 1.58 & 6.930e-01 & 2.06\\ 
\hline\end{tabular} 
\end{table}

%% file: conclusion.tex
\section{Concluding remarks}\label{sec:conc}
In this paper we have combined known shape optimization techniques, with
finite element methods on multiple overlapping meshes. The key features
of this approach has been discussed. For shape optimization problem where the
change of the domain can be parameterized as a translation or rotation, we
observe that the function space of the mesh transformation can be reduced to
constant functions. This yields a big speed-up for updating the mesh, where one
avoids deformation equations or remeshing techniques.
For problems where a part of the domain can be deformed, we have showed that by choosing a meshing describing the part that is changing, one can yield big speed-ups in the mesh deformation step.

Nevertheless, since the MultiMesh FEM is a fairly new method, it has only been explored for time-independent heat and Stokes equation. Further study of Nitsche enforcement of interface conditions is needed to be able to
provide stable finite element schemes for overlapping meshes for other equations.

In conclusion, the results reported in this paper, shows that the combination
of shape optimization holds great promise as a powerful method for avoiding
deformation equations and re-meshing. In a later paper, we will extend this
approach to time-dependent problems, with more complex state-equations.